\documentclass {myjournal}      

\allowdisplaybreaks

\title{Evaluation of some non-elementary  integrals involving the generalized hypergeometric function with some applications}
\headlinetitle{ Some integrals involving the generalized hypergeometric function}
\lastnameone{Nijimbere}
\firstnameone{Victor}
\nameshortone{V.~Nijimbere}
\addressone{School of Mathematics and Statistics, Carleton University, Ottawa, Ontario}
\countryone{Canada}
\emailone{victornijimbere@gmail.com}

\lastnametwo{}
\firstnametwo{}
\nameshorttwo{}
\addresstwo{}
\countrytwo{}
\emailtwo{}

\lastnamethree{}
\firstnamethree{}
\nameshortthree{}
\addressthree{}
\countrythree{}
\emailthree{}

\lastnamefour{}
\firstnamefour{}
\nameshortfour{}
\addressfour{}
\countryfour{}
\emailfour{}

\lastnamefive{}
\firstnamefive{}
\nameshortfive{}
\addressfive{}
\countryfive{}
\emailfive{}

\abstract{The indefinite integral
$$
\int x^\alpha e^{\eta x^\beta}\,_pF_q (a_1, a_2, \cdot\cdot\cdot a_p; b_1, b_2, \cdot\cdot\cdot, b_q; \lambda x^{\gamma})dx,
$$
where $\alpha, \eta, \beta, \lambda, \gamma\ne0$ are real or complex constants and $_pF_q$ is the generalized hypergeometric function,  is evaluated in terms of an infinite series involving the generalized hypergeometric function. 
Related integrals in which the exponential function $e^{\eta x^\beta}$ is either replaced by the hyperbolic function $\cosh\left(\eta x^\beta\right)$ or $\sinh\left(\eta x^\beta\right)$, or the sinusoidal function $\cos\left(\eta x^\beta\right)$ or $\sin\left(\eta x^\beta\right)$, are also evaluated in terms of infinite series involving the generalized hypergeometric function  $_pF_q$. Some application examples from applied analysis, in which some new Fourier and Laplace integrals (or transforms) are evaluated,  are given. The analytical solution of the Orr-Sommerfeld equation (with a linear mean flow background)  in the short-wave limit is expressed in terms of some infinite series involving the hypergeometric series $_2F_3$. Making use of the hyperbolic and Euler identities, some interesting series identities involving exponential, hyperbolic, trigonometric functions and the generalized hypergeometric function are also derived.}



\begin{document}

\section{Introduction}
\label{sec:1}
The generalized hypergeometric function will be extensively used throughout the paper, so its definition is given here for reference.
\begin{definition}The generalized hypergeometric function, denoted as $_pF_q$, is a special function given by the series \cite{AB,NI}
\begin{equation}
_p F_q(a_1, a_2,\cdots,a_p;b_1, b_2, \cdots, b_q; x)=\sum\limits_{n=0}^{\infty}\frac{(a_1)_n (a_2)_n\cdots (a_p)_n}{(b_1)_n (b_2)_n\cdots (b_q)_n}\frac{x^n}{n!},
\label{eq1.1}
\end{equation} 
where $a_1, a_2,\cdots,a_p$ and $;b_1, b_2, \cdots, b_q$ are arbitrary constants, $(\vartheta)_n=\Gamma(\vartheta+n)/\Gamma(\vartheta)$ (Pochhammer's notation \cite{AB}) for any complex $\vartheta$, with $(\vartheta)_0=1$, and $\Gamma$ is the standard gamma function \cite{AB,NI}.
\label{def1}
\end{definition}

The generalized hypergeometric function $_p F_q$ can be used to represent a large class of functions, from simple functions  (e.g. $e^x, \cos (x), \sin(x)$,  and so on) to non-elementary functions (e.g. error functions, incomplete gamma function, and many more) \cite{AB,NI}. Special functions that include Bessel functions, Legendre polynomials,  Hermite polynomials and many others can, as well, be expressed in terms of $_p F_q$ \cite{AB,NI}. A large class of non-elementary integrals (e.g. exponential integral, logarithmic integral, cosine and sine integrals, Dawson's integral, etc), consisting of integrals which cannot neither be evaluated in terms of elementary functions nor in terms of finite algebraic combinations of elementary functions, can also be evaluated in terms of $_p F_q$ \cite{N1,N2,N3,N4}.

Here, the following  complex (more general)  indefinite integral in which the integrand involves a generalized hypergeometric function $_pF_q$
\begin{equation}
\int x^\alpha e^{\eta x^\beta}\,_pF_q (a_1, a_2, \cdot\cdot\cdot a_p; b_1, b_2, \cdot\cdot\cdot, b_q; \lambda x^{\gamma})dx,
\label{eq1.2}
\end{equation}
 and other related  integrals involving hyperbolic and trigonometric functions,
\begin{equation}
\int x^\alpha \cosh\left(\eta x^\beta\right)\,_pF_q (a_1, a_2, \cdot\cdot\cdot a_p; b_1, b_2, \cdot\cdot\cdot, b_q; \lambda x^{\gamma})dx,
\label{eq1.3}
\end{equation}
\begin{equation}
\int x^\alpha \sinh\left(\eta x^\beta\right)\,_pF_q (a_1, a_2, \cdot\cdot\cdot a_p; b_1, b_2, \cdot\cdot\cdot, b_q; \lambda x^{\gamma})dx,
\label{eq1.4}
\end{equation}
\begin{equation}
\int x^\alpha \cos\left(\eta x^\beta\right)\,_pF_q (a_1, a_2, \cdot\cdot\cdot a_p; b_1, b_2, \cdot\cdot\cdot, b_q; \lambda x^{\gamma})dx
\label{eq1.5}
\end{equation}
and
\begin{equation}
\int x^\alpha \sin\left(\eta x^\beta\right)\,_pF_q (a_1, a_2, \cdot\cdot\cdot a_p; b_1, b_2, \cdot\cdot\cdot, b_q; \lambda x^{\gamma})dx,
\label{eq1.6}
\end{equation}
where $\alpha, \eta, \beta, \lambda$ are real or complex arbitrarily constants and $ \gamma$ is an arbitrarily nonzero constant ($\gamma\ne0$), are evaluated. To my knowledge, no one has evaluated  these  integrals before.  In the present paper, they are evaluated in terms of infinite series involving the generalized hypergeometric function $ _pF_q$. Other interesting integrals  involving the generalized hypergeometric functions or series can be found, for example, in \cite{AQ,AR,K,M}.

The advantage in writing these integrals in terms of series involving  the generalized hypergeometric function $_pF_q$ is that for entire functions $_pF_q$ on the complex plane $\mathbb{C}$, the values or the behaviors of $_pF_q$ as $|x|\to\infty$ may be assessed \cite{AB,NI}. In that case, these integrals may find applications in applied analysis and other fields in applied science and engineering.
For instance, if the integrability condition 
\begin{equation}
\int\limits_{-\infty}^{+\infty} |x^\alpha e^{\eta x^\beta}\,_pF_q (a_1, a_2, \cdot\cdot\cdot a_p; b_1, b_2, \cdot\cdot\cdot, b_q; \lambda x^{\gamma})|dx<\infty
\label{eq1.7}
\end{equation}
is satisfied, then a large class of  important integrals in applied analysis  can be evaluated. For example, the Laplace integral (or transform)
\begin{equation}
\mathcal{L}(u)=\int\limits_{0}^{+\infty} e^{-ux} x^\alpha\,_pF_q (a_1, a_2, \cdot\cdot\cdot a_p; b_1, b_2, \cdot\cdot\cdot ,b_q; \lambda x^{\gamma})dx, \, \mbox{Re}(u)>0,
\label{eq1.8}
\end{equation}
and the Fourier integral (or transform)
\begin{equation}
\mathcal{F}(k)=\int\limits_{-\infty}^{+\infty} e^{ikx} x^\alpha\,_pF_q (a_1, a_2, \cdot\cdot\cdot a_p; b_1, b_2, \cdot\cdot\cdot, b_q; \lambda x^{\gamma})dx,
\label{eq1.9}
\end{equation}
 $u$ and $k$ being the Laplace and Fourier parameters respectively and $\mbox{Re}(u)$ the real part of $u$, can be evaluated. Laplace and Fourier integrals
are extensively used in applied analysis to solve different types of differential equations from ordinary differential equations (ODEs) to fractional partial differential equations (FPDEs); they are used in probability theory and statistics to evaluate  moments and characteristic functions, just to name few \cite{B,KS,ML}.

As illustrative applications, the Fourier integral,  $\int_{-\infty}^{+\infty}e^{-\theta^2 x^2} e^{ikx} dx$ which is useful in obtaining the fundamental solution of the diffusion (heat) equation,  is evaluated using (\ref{eq1.9}) in section \ref{sec:3}. This example is chosen for simplification purpose because it is  already known, and therefore, no further investigations such numerical evaluations are needed for comparison. The related Fourier integral involving the Gaussian, $\int_{-\infty}^{+\infty} x^{\alpha} e^{-\theta^2 x^2} e^{ikx} dx, \,\alpha>-2$, is also evaluated. The Laplace integral $\int_{0}^{+\infty} x^{\alpha} e^{-\theta^2 x^2} e^{-ux} dx, \,\alpha>-1$, involving the Gaussian is also evaluated, and hence that of the error function $\mbox{erf}(x)=\int_0^x e^{-v^2}dv$.  The solution of the Orr-Sommerfeld equation (a fourth order ODE used in linear stability theory of fluid flows)  with the plane Couette flow mean background \cite{SH}, examined in \cite{N5}, which may involve the integrals (\ref{eq1.2}), (\ref{eq1.3}) and (\ref{eq1.4}) is also revisited. 



Furthermore, using the hyperbolic identities 
$$ \cosh\left(\eta x^\beta\right)= \left( e^{\eta x^\beta}+ e^{-\eta x^\beta}\right)/2,\,\, \sinh\left(\eta x^\beta\right)= \left( e^{\eta x^\beta}-e^{-\eta x^\beta}\right)/2$$
and Euler's identities 
$$ \cos\left(\eta x^\beta\right)= \left( e^{i\eta x^\beta}+ e^{-i\eta x^\beta}\right)/2,\,\, \sin\left(\eta x^\beta\right)= \left( e^{i\eta x^\beta}-e^{-i\eta x^\beta}\right)/(2i),$$
we are able to obtain some interesting equalities involving infinite series in terms of the generalized hypergeometric function $_p F_q$. Other interesting equalities and series involving the generalized hypergeometric function may be found, for example, in \cite{AR,CR,QQ}.

The paper is organized as follows. The  indefinite integrals (\ref{eq1.2})-(\ref{eq1.6}) are evaluated in section 2. Some new Fourier and Laplace integrals  involving the Gaussian are presented  in section 3, and the Orr-Sommerfeld equation  with the plane Couette flow mean background \cite{SH} is also examined  in section 3. In section 4, a general discussion is given.

\section{Evaluation of the non-elementary integral $\int x^\alpha e^{\eta x^\beta}\,_pF_q (a_1, a_2, \cdot\cdot\cdot a_p; b_1, b_2, \cdot\cdot\cdot, b_q; \lambda x^{\gamma})dx$ and other related non-elementary integrals}
\label{sec:2} 

Let us first prove a lemma which will be used throughout this paper.
\begin{lemma}
Let $ j\ge0, m\ge0$ and $n\ge0$ be integers, and let $\alpha, \beta$ and $\gamma\ne 0$ be constants. Then
\begin{equation}
\prod \limits_{m=0}^{j}(n\gamma+\alpha+m\beta+1)=\prod \limits_{m=0}^{j}(\alpha+m\beta+1)
\frac{\prod \limits_{m=0}^{j}\left(\frac{\alpha+\gamma+m\beta+1}{\gamma}\right)_n}{\prod \limits_{m=0}^{j}\left(\frac{\alpha+m\beta+1}{\gamma}\right)_n}. 
\label{eq2.1}
\end{equation}
\label{lem1}
\end{lemma}
\begin{proof} We use Pochhammer's notation \cite{AB,NI}, see definition \ref{def1}, and the property of the gamma function \cite{AB,NI}, $\Gamma(a+1)=a\Gamma(a)$. Then 
\begin{align}
\prod \limits_{m=0}^{j}(n\gamma+\alpha+m\beta+1)&=\gamma^j \prod \limits_{m=0}^{j}\left(n+\frac{\alpha+m\beta+1}{\gamma}\right)
\nonumber\\ &= \gamma^j \prod \limits_{m=0}^{j}\frac{\Gamma\left(n+\frac{\alpha+m\beta+1}{\gamma}+1\right)}{\Gamma\left(n+\frac{\alpha+m\beta+1}{\gamma}\right)} 
\nonumber\\ &= \gamma^j \prod \limits_{m=0}^{j}\frac{\left(\frac{\alpha+m\beta+1}{\gamma}+1\right)_n\Gamma\left(\frac{\alpha+m\beta+1}{\gamma}+1\right)}{\left(\frac{\alpha+m\beta+1}{\gamma}\right)_n\Gamma\left(\frac{\alpha+m\beta+1}{\gamma}\right)} 
\nonumber\\ &= \gamma^j \prod \limits_{m=0}^{j}\frac{\left(\frac{\alpha+m\beta+1}{\gamma}+1\right)_n}{\left(\frac{\alpha+m\beta+1}{\gamma}\right)_n} \prod \limits_{m=0}^{j} \left(\frac{\alpha+m\beta+1}{\gamma}\right) 
\nonumber \\ & =\prod \limits_{m=0}^{j}(\alpha+m\beta+1)
\frac{\prod \limits_{m=0}^{j}\left(\frac{\alpha+\gamma+m\beta+1}{\gamma}\right)_n}{\prod \limits_{m=0}^{j}\left(\frac{\alpha+m\beta+1}{\gamma}\right)_n}.
\label{eq2.2}
\end{align}
\end{proof}

\begin{proposition} For any constants $\alpha, \beta, \eta, \lambda$ and $\gamma$ any nonzero constant ($\gamma\ne0$),
\begin{multline}
\int x^\alpha e^{\eta x^\beta}\,_pF_q (a_1, a_2, \cdot\cdot\cdot a_p; b_1, b_2, \cdot\cdot\cdot b_q; \lambda x^{\gamma})dx\\=x^{\alpha+1}e^{\eta x^\beta}\sum\limits_{j=0}^{\infty}\frac{\left(-\beta \eta x^\beta\right)^j}{\prod\limits_{m=0}^{j}(\alpha+m\beta+1)}\,_{p+j}F_{q+j} \Bigl(a_1, \cdot\cdot\cdot, a_p, \frac{\alpha+1}{\gamma}, \frac{\alpha+\beta+1}{\gamma}, \cdot\cdot\cdot, \frac{\alpha+j\beta+1}{\gamma}\\; b_1, \cdot\cdot\cdot, b_q,\frac{\alpha+\gamma+1}{\gamma}, \frac{\alpha+\gamma+\beta+1}{\gamma}, \cdot\cdot\cdot, \frac{\alpha+\gamma+j\beta+1}{\gamma}; \lambda x^{\gamma}\Bigr)+C.
\label{eq2.3}
\end{multline}
\label{prp1}
\end{proposition}

\begin{proof} 
The substitution $u^\beta=\eta x^\beta$ and (\ref{eq1.1}) gives
\begin{equation}
\int x^\alpha e^{\eta x^\beta}\,_pF_q (a_1, \cdot\cdot\cdot a_p; b_1, \cdot\cdot\cdot b_q; \lambda x^{\gamma})dx=\sum\limits_{n=0}^{\infty}\frac{(a_1)_n\cdots (a_p)_n}{(b_1)_n\cdots (b_q)_n}\frac{\lambda^n}{n!}
\frac{1}{\eta^{\frac{n\gamma+\alpha+1}{\beta}}}\int u^{n\gamma+\alpha} e^{u^\beta}du.
\label{eq2.4}
\end{equation}
Successive integration by parts that increases the power of $u$ yields
\begin{align}
\int u^{n\gamma+\alpha} e^{u^\beta}du =&\frac{u^{n\gamma+\alpha+1}e^{u^\beta}}{n\gamma+\alpha+\beta+1}-\frac{\beta u^{n\gamma+\alpha+\beta+1}e^{u^\beta}}{(n\gamma+\alpha+1)(n\gamma+\alpha+\beta+1)} \nonumber\\ &+\frac{\beta^2 u^{n\gamma+\alpha+2\beta+1}e^{u^\beta}}{(n\gamma+\alpha+1)(n\gamma+\alpha+\beta+1)(n\gamma+\alpha+2\beta+1)}\nonumber\\ &-\frac{\beta^3  u^{n\gamma+\alpha+3\beta+1}e^{u^\beta}}{(n\gamma+\alpha+1)(n\gamma+\alpha+\beta+1)(n\gamma+\alpha+2\beta+1)(n\gamma+\alpha+3\beta+1)}
\nonumber\\ &+\cdot\cdot\cdot\cdot+\frac{(-1)^j\beta^j  u^{n\gamma+\alpha+j\beta+1}e^{u^\beta}}{\prod\limits_{m=0}^{j}(n\gamma+\alpha+m\beta+1)}+\cdot\cdot\cdot\cdot
\nonumber\\ &=\sum\limits_{j=0}^{\infty}\frac{(-1)^j\beta^j  u^{n\gamma+\alpha+j\beta+1}e^{u^\beta}}{\prod\limits_{m=0}^{j}(n\gamma+\alpha+m\beta+1)}+C.
\label{eq2.5}
\end{align}
Using Lemma \ref{lem1} yields
\begin{equation}
\int u^{n\gamma+\alpha} e^{u^\beta}du = u^{n\gamma+\alpha+1}e^{u^\beta}\sum\limits_{j=0}^{\infty}\frac{\prod \limits_{m=0}^{j}\left(\frac{\alpha+m\beta+1}{\gamma}\right)_n}{\prod \limits_{m=0}^{j}\left(\frac{\alpha+\gamma+m\beta+1}{\gamma}\right)_n}\frac{(-\beta  u^{\beta})^j}{\prod\limits_{m=0}^{j}(\alpha+m\beta+1)}+C.
\label{eq2.6}
\end{equation}
Then,
\begin{multline}
\int x^\alpha e^{\eta x^\beta}\,_pF_q (a_1, \cdot\cdot\cdot a_p; b_1, \cdot\cdot\cdot b_q; \lambda x^{\gamma})dx=
\left(\frac{u}{\eta^{\beta}}\right)^{\alpha+1} e^{u^\beta}\sum\limits_{j=0}^{\infty}\frac{(-\beta  u^{\beta})^j}{\prod\limits_{m=0}^{j}(\alpha+m\beta+1)}\\ \times
 \sum\limits_{n=0}^{\infty}\frac{(a_1)_n\cdots (a_p)_n \left(\frac{\alpha+1}{\gamma}\right)_n\left(\frac{\alpha+\beta+1}{\gamma}\right)_n\cdots \left(\frac{\alpha+j\beta+1}{\gamma}\right)_n }{(b_1)_n\cdots (b_q)_n\left(\frac{\alpha+\gamma+1}{\gamma}\right)_n\left(\frac{\alpha+\gamma+\beta+1}{\gamma}\right)_n\cdots\left(\frac{\alpha+\gamma+j\beta+1}{\gamma}\right)_n}\frac{\left(\frac{ \lambda u^{\gamma}}{\eta^{\frac{\gamma}{\beta}}}\right)^n}{n!}+C\\
=x^{\alpha+1} e^{\eta x^\beta}\sum\limits_{j=0}^{\infty}\frac{(-\beta\eta  x^{\beta})^j}{\prod\limits_{m=0}^{j}(\alpha+m\beta+1)} \\ \times \sum\limits_{n=0}^{\infty}\frac{(a_1)_n\cdots (a_p)_n \left(\frac{\alpha+1}{\gamma}\right)_n\left(\frac{\alpha+\beta+1}{\gamma}\right)_n\cdots \left(\frac{\alpha+j\beta+1}{\gamma}\right)_n }{(b_1)_n\cdots (b_q)_n\left(\frac{\alpha+\gamma+1}{\gamma}\right)_n\left(\frac{\alpha+\gamma+\beta+1}{\gamma}\right)_n\cdots\left(\frac{\alpha+\gamma+j\beta+1}{\gamma}\right)_n}\frac{\left(\lambda x^\gamma\right)^n}{n!}+C,
\label{eq2.7}
\end{multline}
and this gives (\ref{eq2.3}).
\end{proof}
\subsection{Evaluation of some non-elementary integrals involving the hyperbolic functions $\cosh$ and $\sinh$ and the generalized hypergeometric function $_pF_q$}
\label{subsec:2.1}

\begin{proposition} For any constants $\alpha, \beta, \eta, \lambda$ and $\gamma$ any nonzero constant ($\gamma\ne0$),
\begin{multline}
\int x^\alpha \cosh\left(\eta x^\beta\right)\,_pF_q (a_1, a_2, \cdot\cdot\cdot a_p; b_1, b_2, \cdot\cdot\cdot b_q; \lambda x^{\gamma})dx=x^{\alpha+1}\cosh\left(\eta x^\beta\right)\times \\\sum\limits_{j=0}^{\infty}\frac{\left(\beta \eta x^\beta\right)^{2j}}{\prod\limits_{m=0}^{2j}(\alpha+m\beta+1)}\times
\\_{p+j}F_{q+j} \Bigl(a_1, \cdot\cdot\cdot, a_p, \frac{\alpha+1}{\gamma}, \frac{\alpha+\beta+1}{\gamma},\frac{\alpha+2\beta+1}{\gamma}, \cdot\cdot\cdot, \frac{\alpha+2j\beta+1}{\gamma}\\; b_1, \cdot\cdot\cdot, b_q,\frac{\alpha+\gamma+1}{\gamma},\frac{\alpha+\gamma+\beta+1}{\gamma}, \frac{\alpha+\gamma+2\beta+1}{\gamma}, \cdot\cdot\cdot, \frac{\alpha+\gamma+2j\beta+1}{\gamma}; \lambda x^{\gamma}\Bigr)\\-x^{\alpha+1}\sinh\left(\eta x^\beta\right)\sum\limits_{j=0}^{\infty}\frac{\left(\beta \eta x^\beta\right)^{2j+1}}{\prod\limits_{m=0}^{2j+1}(\alpha+m\beta+1)}\times\\_{p+j}F_{q+j} \Bigl(a_1, \cdot\cdot\cdot, a_p, \frac{\alpha+1}{\gamma}, \frac{\alpha+\beta+1}{\gamma}, \frac{\alpha+2\beta+1}{\gamma}, \cdot\cdot\cdot, \frac{\alpha+(2j+1)\beta+1}{\gamma}; b_1, \cdot\cdot\cdot, b_q,\\ \frac{\alpha+\gamma+1}{\gamma}, \frac{\alpha+\gamma+\beta+1}{\gamma}, \frac{\alpha+\gamma+2\beta+1}{\gamma},\cdot\cdot\cdot, \frac{\alpha+\gamma+(2j+1)\beta+1}{\gamma}; \lambda x^{\gamma}\Bigr)+C.
\label{eq2.8} 
\end{multline}
\label{prp2}
\end{proposition}

\begin{proof} 
The change of variable $u^\beta=\eta x^\beta$ and (\ref{eq1.1}) yields
\begin{multline}
\int x^\alpha \cosh\left(\eta x^\beta\right)\,_pF_q (a_1, \cdot\cdot\cdot a_p; b_1, \cdot\cdot\cdot b_q; \lambda x^{\gamma})dx\\=\sum\limits_{n=0}^{\infty}\frac{(a_1)_n\cdots (a_p)_n}{(b_1)_n\cdots (b_q)_n}\frac{\lambda^n}{n!}
\frac{1}{\eta^{\frac{n\gamma+\alpha+1}{\beta}}}\int u^{n\gamma+\alpha} \cosh\left(u^\beta\right)du+C.
\label{eq2.9}
\end{multline}
Successive integration by parts that increases the power of $u$ gives
\begin{align}
&\int u^{n\gamma+\alpha} \cosh\left(u^\beta\right)du =\frac{u^{n\gamma+\alpha+1}\cosh\left(u^\beta\right)}{n\gamma+\alpha+\beta+1}-\frac{\beta u^{n\gamma+\alpha+\beta+1}\sinh\left(u^\beta\right)}{(n\gamma+\alpha+1)(n\gamma+\alpha+\beta+1)} \nonumber\\ &+\frac{\beta^2 u^{n\gamma+\alpha+2\beta+1}\cosh\left(u^\beta\right)}{(n\gamma+\alpha+1)(n\gamma+\alpha+\beta+1)(n\gamma+\alpha+2\beta+1)}\nonumber\\ &-\frac{\beta^3  u^{n\gamma+\alpha+3\beta+1}\sinh\left(u^\beta\right)}{(n\gamma+\alpha+1)(n\gamma+\alpha+\beta+1)(n\gamma+\alpha+2\beta+1)(n\gamma+\alpha+3\beta+1)}
\nonumber\\ &+\cdot\cdot+\frac{\beta^{2j}  u^{n\gamma+\alpha+2j\beta+1}\cosh\left(u^\beta\right)}{\prod\limits_{m=0}^{2j}(n\gamma+\alpha+m\beta+1)}+\cdots-\frac{\beta^{2j+1}  u^{n\gamma+\alpha+(2j+1)\beta+1}\sinh\left(u^\beta\right)}{\prod\limits_{m=0}^{2j+1}(n\gamma+\alpha+m\beta+1)}-\cdot\cdot
\nonumber\\ &=\cosh\left(u^\beta\right)\sum\limits_{j=0}^{\infty}\frac{\beta^{2j}  u^{n\gamma+\alpha+2j\beta+1}}{\prod\limits_{m=0}^{2j}(n\gamma+\alpha+m\beta+1)}-\sinh\left(u^\beta\right)\sum\limits_{j=0}^{\infty}\frac{\beta^{2j+1}  u^{n\gamma+\alpha+(2j+1)\beta+1}}{\prod\limits_{m=0}^{2j+1}(n\gamma+\alpha+m\beta+1)}+C.
\label{eq2.10}
\end{align}
Using Lemma \ref{lem1} yields
\begin{multline}
\int u^{n\gamma+\alpha}\cosh\left(u^\beta\right)du = u^{n\gamma+\alpha+1}\cosh\left(u^\beta\right)\sum\limits_{j=0}^{\infty}\frac{\prod \limits_{m=0}^{2j}\left(\frac{\alpha+m\beta+1}{\gamma}\right)_n}{\prod \limits_{m=0}^{2j}\left(\frac{\alpha+\gamma+m\beta+1}{\gamma}\right)_n}\frac{(\beta  u^{\beta})^{2j}}{\prod\limits_{m=0}^{2j}(\alpha+m\beta+1)}\\ -u^{n\gamma+\alpha+1}\sinh\left(u^\beta\right)\sum\limits_{j=0}^{\infty}\frac{\prod \limits_{m=0}^{2j+1}\left(\frac{\alpha+m\beta+1}{\gamma}\right)_n}{\prod \limits_{m=0}^{2j+1}\left(\frac{\alpha+\gamma+m\beta+1}{\gamma}\right)_n}\frac{(\beta  u^{\beta})^{2j+1}}{\prod\limits_{k=0}^{2j+1}(\alpha+m\beta+1)}+C.
\label{eq2.11}
\end{multline}
Then,
\begin{multline}
\int x^\alpha\cosh\left(\eta x^\beta\right)\,_pF_q (a_1, \cdot\cdot\cdot a_p; b_1, \cdot\cdot\cdot b_q; \lambda x^{\gamma})dx
=\left(\frac{u}{\eta^{\beta}}\right)^{\alpha+1}\cosh\left(u^\beta\right)\\ \times\sum\limits_{j=0}^{\infty}\frac{(\beta  u^{\beta})^{2j}}{\prod\limits_{m=0}^{2j}(\alpha+m\beta+1)}\times \\ \sum\limits_{n=0}^{\infty}\frac{(a_1)_n\cdot\cdot (a_p)_n \left(\frac{\alpha+1}{\gamma}\right)_n\left(\frac{\alpha+\beta+1}{\gamma}\right)_n\cdot\cdot \left(\frac{\alpha+2j\beta+1}{\gamma}\right)_n }{(b_1)_n\cdot\cdot (b_q)_n\left(\frac{\alpha+\gamma+1}{\gamma}\right)_n\left(\frac{\alpha+\gamma+\beta+1}{\gamma}\right)_n\cdot\cdot\left(\frac{\alpha+\gamma+2j\beta+1}{\gamma}\right)_n}\frac{\left(\frac{ \lambda u^{\gamma}}{\eta^{\frac{\gamma}{\beta}}}\right)^n}{n!}\\
-\left(\frac{u}{\eta^{\beta}}\right)^{\alpha+1}\sinh\left(u^\beta\right) \sum\limits_{j=0}^{\infty}\frac{(\beta  u^{\beta})^{2j+1}}{\prod\limits_{m=0}^{2j+1}(\alpha+m\beta+1)} \times \\ \sum\limits_{n=0}^{\infty}\frac{(a_1)_n\cdot\cdot (a_p)_n \left(\frac{\alpha+1}{\gamma}\right)_n\left(\frac{\alpha+\beta+1}{\gamma}\right)_n\cdot\cdot \left(\frac{\alpha+(2j+1)\beta+1}{\gamma}\right)_n }{(b_1)_n\cdot\cdot (b_q)_n\left(\frac{\alpha+\gamma+1}{\gamma}\right)_n\left(\frac{\alpha+\gamma+\beta+1}{\gamma}\right)_n\cdot\cdot\left(\frac{\alpha+\gamma+(2j+1)\beta+1}{\gamma}\right)_n}\frac{\left(\frac{ \lambda u^{\gamma}}{\eta^{\frac{\gamma}{\beta}}}\right)^n}{n!}\\
=x^{\alpha+1}\cosh\left(\eta x^\beta\right)\sum\limits_{j=0}^{\infty}\frac{(\beta  \eta x^{\beta})^{2j}}{\prod\limits_{m=0}^{2j}(\alpha+m\beta+1)}\\ \times \sum\limits_{n=0}^{\infty}\frac{(a_1)_n\cdot\cdot (a_p)_n \left(\frac{\alpha+1}{\gamma}\right)_n\left(\frac{\alpha+\beta+1}{\gamma}\right)_n\cdot\cdot \left(\frac{\alpha+2j\beta+1}{\gamma}\right)_n }{(b_1)_n\cdot\cdot (b_q)_n\left(\frac{\alpha+\gamma+1}{\gamma}\right)_n\left(\frac{\alpha+\gamma+\beta+1}{\gamma}\right)_n\cdot\cdot\left(\frac{\alpha+\gamma+2j\beta+1}{\gamma}\right)_n}\frac{\left(\lambda x^\gamma\right)^n}{n!}\\
-x^{\alpha+1}\sinh\left(\eta x^\beta\right) \sum\limits_{j=0}^{\infty}\frac{(\beta \eta x^{\beta})^{2j+1}}{\prod\limits_{m=0}^{2j+1}(\alpha+m\beta+1)}\times \\ \sum\limits_{n=0}^{\infty}\frac{(a_1)_n\cdot\cdot (a_p)_n \left(\frac{\alpha+1}{\gamma}\right)_n\left(\frac{\alpha+\beta+1}{\gamma}\right)_n\cdot\cdot \left(\frac{\alpha+(2j+1)\beta+1}{\gamma}\right)_n }{(b_1)_n\cdot\cdot (b_q)_n\left(\frac{\alpha+\gamma+1}{\gamma}\right)_n\left(\frac{\alpha+\gamma+\beta+1}{\gamma}\right)_n\cdot\cdot\left(\frac{\alpha+\gamma+(2j+1)\beta+1}{\gamma}\right)_n}\frac{\left(\lambda x^\gamma\right)^n}{n!}+C,
\label{eq2.12}
\end{multline}
and hence, (\ref{eq2.12}) gives (\ref{eq2.8}).
\end{proof}

\begin{proposition} For any constants $\alpha, \beta, \eta, \lambda$ and $\gamma$ any nonzero constant ($\gamma\ne0$),
\begin{multline}
\int x^\alpha \sinh\left(\eta x^\beta\right)\,_pF_q (a_1, a_2, \cdot\cdot\cdot a_p; b_1, b_2, \cdot\cdot\cdot b_q; \lambda x^{\gamma})dx\\=x^{\alpha+1}\sinh\left(\eta x^\beta\right)\sum\limits_{j=0}^{\infty}\frac{\left(\beta \eta x^\beta\right)^{2j}}{\prod\limits_{m=0}^{2j}(\alpha+m\beta+1)}\\ \times _{p+j}F_{q+j} \Bigl(a_1, \cdot\cdot\cdot, a_p, \frac{\alpha+1}{\gamma}, \frac{\alpha+\beta+1}{\gamma},\frac{\alpha+2\beta+1}{\gamma}, \cdot\cdot\cdot, \frac{\alpha+2j\beta+1}{\gamma}\\; b_1, \cdot\cdot\cdot, b_q,\frac{\alpha+\gamma+1}{\gamma},\frac{\alpha+\gamma+\beta+1}{\gamma}, \frac{\alpha+\gamma+2\beta+1}{\gamma}, \cdot\cdot\cdot, \frac{\alpha+\gamma+2j\beta+1}{\gamma}; \lambda x^{\gamma}\Bigr)\\-x^{\alpha+1}\cosh\left(\eta x^\beta\right)\sum\limits_{j=0}^{\infty}\frac{\left(\beta \eta x^\beta\right)^{2j+1}}{\prod\limits_{m=0}^{2j+1}(\alpha+m\beta+1)}\times\\ _{p+j}F_{q+j} \Bigl(a_1, \cdot\cdot\cdot, a_p, \frac{\alpha+1}{\gamma}, \frac{\alpha+\beta+1}{\gamma}, \frac{\alpha+2\beta+1}{\gamma}, \cdot\cdot\cdot, \frac{\alpha+(2j+1)\beta+1}{\gamma};  b_1, \cdot\cdot\cdot, b_q,\\ \frac{\alpha+\gamma+1}{\gamma}, \frac{\alpha+\gamma+\beta+1}{\gamma}, \frac{\alpha+\gamma+2\beta+1}{\gamma},\cdot\cdot\cdot, \frac{\alpha+\gamma+(2j+1)\beta+1}{\gamma}; \lambda x^{\gamma}\Bigr)+C.
\label{eq2.13}
\end{multline}
\label{prp3}
\end{proposition}

\begin{proof} 
Making the change of variable $u^\beta=\eta x^\beta$  and using  (\ref{eq1.1}) yields
\begin{multline}
\int x^\alpha \sinh\left(\eta x^\beta\right)\,_pF_q (a_1, \cdot\cdot\cdot a_p; b_1, \cdot\cdot\cdot b_q; \lambda x^{\gamma})dx\\=\sum\limits_{n=0}^{\infty}\frac{(a_1)_n\cdots (a_p)_n}{(b_1)_n\cdots (b_q)_n}\frac{\lambda^n}{n!}
\frac{1}{\eta^{\frac{n\gamma+\alpha+1}{\beta}}}\int u^{n\gamma+\alpha} \sinh\left(u^\beta\right)du+C.
\label{eq2.14}
\end{multline}
Successive integration by parts that increase the power of $u$ gives
\begin{align}
&\int u^{n\gamma+\alpha} \sinh\left(u^\beta\right)du =\frac{u^{n\gamma+\alpha+1}\sinh\left(u^\beta\right)}{n\gamma+\alpha+\beta+1}-\frac{\beta u^{n\gamma+\alpha+\beta+1}\cosh\left(u^\beta\right)}{(n\gamma+\alpha+1)(n\gamma+\alpha+\beta+1)} \nonumber\\ &+\frac{\beta^2 u^{n\gamma+\alpha+2\beta+1}\sinh\left(u^\beta\right)}{(n\gamma+\alpha+1)(n\gamma+\alpha+\beta+1)(n\gamma+\alpha+2\beta+1)}\nonumber\\ &-\frac{\beta^3  u^{n\gamma+\alpha+3\beta+1}\cosh\left(u^\beta\right)}{(n\gamma+\alpha+1)(n\gamma+\alpha+\beta+1)(n\gamma+\alpha+2\beta+1)(n\gamma+\alpha+3\beta+1)}
\nonumber\\ &+\cdot\cdot+\frac{\beta^{2j}  u^{n\gamma+\alpha+2j\beta+1}\sinh\left(u^\beta\right)}{\prod\limits_{m=0}^{2j}(n\gamma+\alpha+m\beta+1)}+\cdot\cdot-\frac{\beta^{2j+1}  u^{n\gamma+\alpha+(2j+1)\beta+1}\cosh\left(u^\beta\right)}{\prod\limits_{m=0}^{2j+1}(n\gamma+\alpha+m\beta+1)}-\cdot\cdot
\nonumber\\ &=\sinh\left(u^\beta\right)\sum\limits_{j=0}^{\infty}\frac{\beta^{2j}  u^{n\gamma+\alpha+2j\beta+1}}{\prod\limits_{m=0}^{2j}(n\gamma+\alpha+m\beta+1)}-\cosh\left(u^\beta\right)\sum\limits_{j=0}^{\infty}\frac{\beta^{2j+1}  u^{n\gamma+\alpha+(2j+1)\beta+1}}{\prod\limits_{m=0}^{2j+1}(n\gamma+\alpha+m\beta+1)}+C.
\label{eq2.15}
\end{align}
Using Lemma \ref{lem1} yields
\begin{multline}
\int u^{n\gamma+\alpha}\cosh\left(u^\beta\right)du = u^{n\gamma+\alpha+1}\sinh\left(u^\beta\right)\sum\limits_{j=0}^{\infty}\frac{\prod \limits_{m=0}^{2j}\left(\frac{\alpha+m\beta+1}{\gamma}\right)_n}{\prod \limits_{m=0}^{2j}\left(\frac{\alpha+\gamma+m\beta+1}{\gamma}\right)_n}\frac{(\beta  u^{\beta})^{2j}}{\prod\limits_{m=0}^{2j}(\alpha+m\beta+1)}\\ -u^{n\gamma+\alpha+1}\cosh\left(u^\beta\right)\sum\limits_{j=0}^{\infty}\frac{\prod \limits_{m=0}^{2j+1}\left(\frac{\alpha+m\beta+1}{\gamma}\right)_n}{\prod \limits_{m=0}^{2j+1}\left(\frac{\alpha+\gamma+m\beta+1}{\gamma}\right)_n}\frac{(\beta  u^{\beta})^{2j+1}}{\prod\limits_{m=0}^{2j+1}(\alpha+m\beta+1)}+C.
\label{eq2.16}
\end{multline}
Then,
\begin{multline}
\int x^\alpha\sinh\left(\eta x^\beta\right)\,_pF_q (a_1, \cdot\cdot\cdot a_p; b_1, \cdot\cdot\cdot b_q; \lambda x^{\gamma})dx
=\left(\frac{u}{\eta^{\beta}}\right)^{\alpha+1}\sinh\left(u^\beta\right)\\ \times\sum\limits_{j=0}^{\infty}\frac{(\beta  u^{\beta})^{2j}}{\prod\limits_{m=0}^{2j}(\alpha+m\beta+1)}\\ \times \sum\limits_{n=0}^{\infty}\frac{(a_1)_n\cdot\cdot (a_p)_n \left(\frac{\alpha+1}{\gamma}\right)_n\left(\frac{\alpha+\beta+1}{\gamma}\right)_n\cdot\cdot \left(\frac{\alpha+2j\beta+1}{\gamma}\right)_n }{(b_1)_n\cdot\cdot (b_q)_n\left(\frac{\alpha+\gamma+1}{\gamma}\right)_n\left(\frac{\alpha+\gamma+\beta+1}{\gamma}\right)_n\cdot\cdot\left(\frac{\alpha+\gamma+2j\beta+1}{\gamma}\right)_n}\frac{\left(\frac{ \lambda u^{\gamma}}{\eta^{\frac{\gamma}{\beta}}}\right)^n}{n!}\\
-\left(\frac{u}{\eta^{\beta}}\right)^{\alpha+1}\cosh\left(u^\beta\right) \sum\limits_{j=0}^{\infty}\frac{(\beta  u^{\beta})^{2j+1}}{\prod\limits_{m=0}^{2j+1}(\alpha+m\beta+1)}\\ \times\sum\limits_{n=0}^{\infty}\frac{(a_1)_n\cdot\cdot (a_p)_n \left(\frac{\alpha+1}{\gamma}\right)_n\left(\frac{\alpha+\beta+1}{\gamma}\right)_n\cdot\cdot \left(\frac{\alpha+(2j+1)\beta+1}{\gamma}\right)_n }{(b_1)_n\cdot\cdot (b_q)_n\left(\frac{\alpha+\gamma+1}{\gamma}\right)_n\left(\frac{\alpha+\gamma+\beta+1}{\gamma}\right)_n\cdot\cdot\left(\frac{\alpha+\gamma+(2j+1)\beta+1}{\gamma}\right)_n}\frac{\left(\frac{ \lambda u^{\gamma}}{\eta^{\frac{\gamma}{\beta}}}\right)^n}{n!}\\
=x^{\alpha+1}\sinh\left(\eta x^\beta\right)\sum\limits_{j=0}^{\infty}\frac{(\beta \eta x^{\beta})^{2j}}{\prod\limits_{m=0}^{2j}(\alpha+m\beta+1)}\\ \times \sum\limits_{n=0}^{\infty}\frac{(a_1)_n\cdot\cdot (a_p)_n \left(\frac{\alpha+1}{\gamma}\right)_n\left(\frac{\alpha+\beta+1}{\gamma}\right)_n\cdot\cdot \left(\frac{\alpha+2j\beta+1}{\gamma}\right)_n }{(b_1)_n\cdot\cdot (b_q)_n\left(\frac{\alpha+\gamma+1}{\gamma}\right)_n\left(\frac{\alpha+\gamma+\beta+1}{\gamma}\right)_n\cdot\cdot\left(\frac{\alpha+\gamma+2j\beta+1}{\gamma}\right)_n}\frac{\left(\lambda x^\gamma\right)^n}{n!}\\
-x^{\alpha+1}\cosh\left(\eta x^\beta\right) \sum\limits_{j=0}^{\infty}\frac{(\beta  \eta x^{\beta})^{2j+1}}{\prod\limits_{m=0}^{2j+1}(\alpha+m\beta+1)} \\ \times\sum\limits_{n=0}^{\infty}\frac{(a_1)_n\cdot\cdot (a_p)_n \left(\frac{\alpha+1}{\gamma}\right)_n\left(\frac{\alpha+\beta+1}{\gamma}\right)_n\cdot\cdot \left(\frac{\alpha+(2j+1)\beta+1}{\gamma}\right)_n }{(b_1)_n\cdot\cdot (b_q)_n\left(\frac{\alpha+\gamma+1}{\gamma}\right)_n\left(\frac{\alpha+\gamma+\beta+1}{\gamma}\right)_n\cdot\cdot\left(\frac{\alpha+\gamma+(2j+1)\beta+1}{\gamma}\right)_n}\frac{\left(\lambda x^\gamma\right)^n}{n!}+C,
\label{eq2.17}
\end{multline}
and hence, (\ref{eq2.17}) gives (\ref{eq2.13}).
\end{proof}

\begin{theorem} For any constants $\alpha, \beta, \eta, \lambda$ and $\gamma$ any nonzero constant ($\gamma\ne0$),
\begin{multline}
\cosh\left(\eta x^\beta\right)\sum\limits_{j=0}^{\infty}\frac{\left(\beta \eta x^\beta\right)^{2j}}{\prod\limits_{m=0}^{2j}(\alpha+m\beta+1)}\times \\ \,_{p+j}F_{q+j} \Bigl(a_1, \cdot\cdot\cdot, a_p, \frac{\alpha+1}{\gamma}, \frac{\alpha+\beta+1}{\gamma},\frac{\alpha+2\beta+1}{\gamma}, \cdot\cdot\cdot, \frac{\alpha+2j\beta+1}{\gamma}\\; b_1, \cdot\cdot\cdot, b_q,\frac{\alpha+\gamma+1}{\gamma},\frac{\alpha+\gamma+\beta+1}{\gamma}, \frac{\alpha+\gamma+2\beta+1}{\gamma}, \cdot\cdot\cdot, \frac{\alpha+\gamma+2j\beta+1}{\gamma}; \lambda x^{\gamma}\Bigr)\\-\sinh\left(\eta x^\beta\right)\sum\limits_{j=0}^{\infty}\frac{\left(\beta \eta x^\beta\right)^{2j+1}}{\prod\limits_{m=0}^{2j+1}(\alpha+m\beta+1)}\times 
\\ _{p+j}F_{q+j} \Bigl(a_1, \cdot\cdot\cdot, a_p, \frac{\alpha+1}{\gamma}, \frac{\alpha+\beta+1}{\gamma}, \frac{\alpha+2\beta+1}{\gamma}, \cdot\cdot\cdot,\frac{\alpha+(2j+1)\beta+1}{\gamma}; b_1, \cdot\cdot\cdot, b_q,\\ \frac{\alpha+\gamma+1}{\gamma}, \frac{\alpha+\gamma+\beta+1}{\gamma}, \frac{\alpha+\gamma+2\beta+1}{\gamma},\cdot\cdot\cdot, \frac{\alpha+\gamma+(2j+1)\beta+1}{\gamma}; \lambda x^{\gamma}\Bigr).\\
=\frac{1}{2}\Bigl[e^{\eta x^\beta}\sum\limits_{j=0}^{\infty}\frac{\left(-\beta \eta x^\beta\right)^j}{\prod\limits_{m=0}^{j}(\alpha+m\beta+1)}\,_{p+j}F_{q+j} \Bigl(a_1, \cdot\cdot\cdot, a_p, \frac{\alpha+1}{\gamma}, \frac{\alpha+\beta+1}{\gamma}, \cdot\cdot\cdot, \frac{\alpha+j\beta+1}{\gamma}\\; b_1, \cdot\cdot\cdot, b_q,\frac{\alpha+\gamma+1}{\gamma}, \frac{\alpha+\gamma+\beta+1}{\gamma}, \cdot\cdot\cdot, \frac{\alpha+\gamma+j\beta+1}{\gamma}; \lambda x^{\gamma}\bigr)\\
\\+e^{-\eta x^\beta}\sum\limits_{j=0}^{\infty}\frac{\left(\beta \eta x^\beta\right)^j}{\prod\limits_{m=0}^{j}(\alpha+m\beta+1)}\,_{p+j}F_{q+j} \Bigl(a_1, \cdot\cdot\cdot, a_p, \frac{\alpha+1}{\gamma}, \frac{\alpha+\beta+1}{\gamma}, \cdot\cdot\cdot, \frac{\alpha+j\beta+1}{\gamma}\\; b_1, \cdot\cdot\cdot, b_q,\frac{\alpha+\gamma+1}{\gamma}, \frac{\alpha+\gamma+\beta+1}{\gamma}, \cdot\cdot\cdot, \frac{\alpha+\gamma+j\beta+1}{\gamma}; \lambda x^{\gamma}\Bigr)\Bigr].
\label{eq2.18}
\end{multline}
\label{th1}
\end{theorem}

\begin{proof}
Using the hyperbolic identity $ \cosh\left(\eta x^\beta\right)= \left( e^{\eta x^\beta}+ e^{-\eta x^\beta}\right)/2$ and Proposition \ref{prp1} yields
\begin{multline}
\int x^\alpha\cosh\left(\eta x^\beta\right)\,_pF_q (a_1, \cdot\cdot\cdot a_p; b_1, \cdot\cdot\cdot b_q; \lambda x^{\gamma})dx=\frac{1}{2}\Bigl[\\ \int x^\alpha e^{\eta x^\beta}\,_pF_q (a_1, \cdot\cdot\cdot a_p; b_1, \cdot\cdot\cdot b_q; \lambda x^{\gamma})dx+\int x^\alpha e^{-\eta x^\beta}\,_pF_q (a_1, \cdot\cdot\cdot a_p; b_1, \cdot\cdot\cdot b_q; \lambda x^{\gamma})dx\Bigr]\\
=\frac{1}{2}\Bigl[e^{\eta x^\beta}\sum\limits_{j=0}^{\infty}\frac{\left(-\beta \eta x^\beta\right)^j}{\prod\limits_{m=0}^{j}(\alpha+m\beta+1)}\,_{p+j}F_{q+j} \Bigl(a_1, \cdot\cdot\cdot, a_p, \frac{\alpha+1}{\gamma}, \frac{\alpha+\beta+1}{\gamma}, \cdot\cdot\cdot, \frac{\alpha+j\beta+1}{\gamma}\\; b_1, \cdot\cdot\cdot, b_q,\frac{\alpha+\gamma+1}{\gamma}, \frac{\alpha+\gamma+\beta+1}{\gamma}, \cdot\cdot\cdot, \frac{\alpha+\gamma+j\beta+1}{\gamma}; \lambda x^{\gamma}\bigr)\\
\\+e^{-\eta x^\beta}\sum\limits_{j=0}^{\infty}\frac{\left(\beta \eta x^\beta\right)^j}{\prod\limits_{m=0}^{j}(\alpha+m\beta+1)}\,_{p+j}F_{q+j} \Bigl(a_1, \cdot\cdot\cdot, a_p, \frac{\alpha+1}{\gamma}, \frac{\alpha+\beta+1}{\gamma}, \cdot\cdot\cdot, \frac{\alpha+j\beta+1}{\gamma}\\; b_1, \cdot\cdot\cdot, b_q,\frac{\alpha+\gamma+1}{\gamma}, \frac{\alpha+\gamma+\beta+1}{\gamma}, \cdot\cdot\cdot, \frac{\alpha+\gamma+j\beta+1}{\gamma}; \lambda x^{\gamma}\Bigr)\Bigr]+C.
\label{eq2.19}
\end{multline}
Hence, Comparing (\ref{eq2.19}) with (\ref{eq2.8}) gives (\ref{eq2.18}). 
\end{proof}

\begin{theorem} For any constants $\alpha, \beta, \eta, \lambda$ and $\gamma$ any nonzero constant ($\gamma\ne0$),
\begin{multline}
\sinh\left(\eta x^\beta\right)\sum\limits_{j=0}^{\infty}\frac{\left(\beta \eta x^\beta\right)^{2j}}{\prod\limits_{m=0}^{2j}(\alpha+m\beta+1)}\times \\ _{p+j}F_{q+j} \Bigl(a_1, \cdot\cdot\cdot, a_p, \frac{\alpha+1}{\gamma}, \frac{\alpha+\beta+1}{\gamma},\frac{\alpha+2\beta+1}{\gamma}, \cdot\cdot\cdot, \frac{\alpha+2j\beta+1}{\gamma}\\; b_1, \cdot\cdot\cdot, b_q,\frac{\alpha+\gamma+1}{\gamma},\frac{\alpha+\gamma+\beta+1}{\gamma}, \frac{\alpha+\gamma+2\beta+1}{\gamma}, \cdot\cdot\cdot, \frac{\alpha+\gamma+2j\beta+1}{\gamma}; \lambda x^{\gamma}\Bigr)\\
-\cosh\left(\eta x^\beta\right)\sum\limits_{j=0}^{\infty}\frac{\left(\beta \eta x^\beta\right)^{2j+1}}{\prod\limits_{m=0}^{2j+1}(\alpha+m\beta+1)}\times \\ _{p+j}F_{q+j} \Bigl(a_1, \cdot\cdot\cdot, a_p, \frac{\alpha+1}{\gamma}, \frac{\alpha+\beta+1}{\gamma}, \frac{\alpha+2\beta+1}{\gamma}, \cdot\cdot\cdot, \frac{\alpha+(2j+1)\beta+1}{\gamma}; b_1, \cdot\cdot\cdot, b_q,\\ \frac{\alpha+\gamma+1}{\gamma}, \frac{\alpha+\gamma+\beta+1}{\gamma}, \frac{\alpha+\gamma+2\beta+1}{\gamma},\cdot\cdot\cdot, \frac{\alpha+\gamma+(2j+1)\beta+1}{\gamma}; \lambda x^{\gamma}\Bigr).\\ 
=\frac{1}{2}\Bigl[e^{\eta x^\beta}\sum\limits_{j=0}^{\infty}\frac{\left(-\beta \eta x^\beta\right)^j}{\prod\limits_{m=0}^{j}(\alpha+m\beta+1)}\,_{p+j}F_{q+j} \Bigl(a_1, \cdot\cdot\cdot, a_p, \frac{\alpha+1}{\gamma}, \frac{\alpha+\beta+1}{\gamma}, \cdot\cdot\cdot, \frac{\alpha+j\beta+1}{\gamma}\\; b_1, \cdot\cdot\cdot, b_q,\frac{\alpha+\gamma+1}{\gamma}, \frac{\alpha+\gamma+\beta+1}{\gamma}, \cdot\cdot\cdot, \frac{\alpha+\gamma+j\beta+1}{\gamma}; \lambda x^{\gamma}\Bigr)\\
-e^{-\eta x^\beta}\sum\limits_{j=0}^{\infty}\frac{\left(\beta \eta x^\beta\right)^j}{\prod\limits_{m=0}^{j}(\alpha+m\beta+1)}\,_{p+j}F_{q+j} \Bigl(a_1, \cdot\cdot\cdot, a_p, \frac{\alpha+1}{\gamma}, \frac{\alpha+\beta+1}{\gamma}, \cdot\cdot\cdot, \frac{\alpha+j\beta+1}{\gamma}\\; b_1, \cdot\cdot\cdot, b_q,\frac{\alpha+\gamma+1}{\gamma}, \frac{\alpha+\gamma+\beta+1}{\gamma}, \cdot\cdot\cdot, \frac{\alpha+\gamma+j\beta+1}{\gamma}; \lambda x^{\gamma}\Bigr)\Bigr].
\label{eq2.20}
\end{multline}
\label{th2}
\end{theorem}

\begin{proof}
To prove (\ref{eq2.20}), we use the hyperbolic identity $\sinh\left(\eta x^\beta\right)= \left( e^{\eta x^\beta}-e^{-\eta x^\beta}\right)/2$ and Proposition \ref{prp1}, and  we obtain
\begin{multline}
\int x^\alpha\sinh\left(\eta x^\beta\right)\,_pF_q (a_1, \cdot\cdot\cdot a_p; b_1, \cdot\cdot\cdot b_q; \lambda x^{\gamma})dx= \frac{1}{2}\Bigl[\\ \int x^\alpha e^{\eta x^\beta}\,_pF_q (a_1, \cdot\cdot\cdot a_p; b_1, \cdot\cdot\cdot b_q; \lambda x^{\gamma})dx+\int x^\alpha e^{-\eta x^\beta}\,_pF_q (a_1, \cdot\cdot\cdot a_p; b_1, \cdot\cdot\cdot b_q; \lambda x^{\gamma})dx\Bigr]\\
=\frac{1}{2}\Bigl[e^{\eta x^\beta}\sum\limits_{j=0}^{\infty}\frac{\left(-\beta \eta x^\beta\right)^j}{\prod\limits_{m=0}^{j}(\alpha+m\beta+1)}\,_{p+j}F_{q+j} \Bigl(a_1, \cdot\cdot\cdot, a_p, \frac{\alpha+1}{\gamma}, \frac{\alpha+\beta+1}{\gamma}, \cdot\cdot\cdot, \frac{\alpha+j\beta+1}{\gamma}\\; b_1, \cdot\cdot\cdot, b_q,\frac{\alpha+\gamma+1}{\gamma}, \frac{\alpha+\gamma+\beta+1}{\gamma}, \cdot\cdot\cdot, \frac{\alpha+\gamma+j\beta+1}{\gamma}; \lambda x^{\gamma}\bigr)\\
\\-e^{-\eta x^\beta}\sum\limits_{j=0}^{\infty}\frac{\left(\beta \eta x^\beta\right)^j}{\prod\limits_{m=0}^{j}(\alpha+m\beta+1)}\,_{p+j}F_{q+j} \Bigl(a_1, \cdot\cdot\cdot, a_p, \frac{\alpha+1}{\gamma}, \frac{\alpha+\beta+1}{\gamma}, \cdot\cdot\cdot, \frac{\alpha+j\beta+1}{\gamma}\\; b_1, \cdot\cdot\cdot, b_q,\frac{\alpha+\gamma+1}{\gamma}, \frac{\alpha+\gamma+\beta+1}{\gamma}, \cdot\cdot\cdot, \frac{\alpha+\gamma+j\beta+1}{\gamma}; \lambda x^{\gamma}\Bigr)\Bigr]+C.
\label{eq2.21}
\end{multline}
Hence, Comparing (\ref{eq2.21}) with (\ref{eq2.13}) gives (\ref{eq2.20}). 
\end{proof}

\begin{theorem} For any constants $\alpha, \beta, \eta, \lambda$ and $\gamma$ any nonzero constant ($\gamma\ne0$),
\begin{multline}
e^{\eta x^\beta}\sum\limits_{j=0}^{\infty}\frac{\left(-\beta \eta x^\beta\right)^j}{\prod\limits_{m=0}^{j}(\alpha+m\beta+1)}\,_{p+j}F_{q+j} \Bigl(a_1, \cdot\cdot\cdot, a_p, \frac{\alpha+1}{\gamma}, \frac{\alpha+\beta+1}{\gamma}, \cdot\cdot\cdot, \frac{\alpha+j\beta+1}{\gamma}\\; b_1, \cdot\cdot\cdot, b_q,\frac{\alpha+\gamma+1}{\gamma}, \frac{\alpha+\gamma+\beta+1}{\gamma}, \cdot\cdot\cdot, \frac{\alpha+\gamma+j\beta+1}{\gamma}; \lambda x^{\gamma}\Bigr)\\
=\cosh\left(\eta x^\beta\right)\sum\limits_{j=0}^{\infty}\frac{\left(\beta \eta x^\beta\right)^{2j}}{\prod\limits_{m=0}^{2j}(\alpha+m\beta+1)}\times \\ _{p+j}F_{q+j} \Bigl(a_1, \cdot\cdot\cdot, a_p, \frac{\alpha+1}{\gamma}, \frac{\alpha+\beta+1}{\gamma},\frac{\alpha+2\beta+1}{\gamma}, \cdot\cdot\cdot, \frac{\alpha+2j\beta+1}{\gamma}\\; b_1, \cdot\cdot\cdot, b_q,\frac{\alpha+\gamma+1}{\gamma},\frac{\alpha+\gamma+\beta+1}{\gamma}, \frac{\alpha+\gamma+2\beta+1}{\gamma}, \cdot\cdot\cdot, \frac{\alpha+\gamma+2j\beta+1}{\gamma}; \lambda x^{\gamma}\Bigr)\\
-\sinh\left(\eta x^\beta\right)\sum\limits_{j=0}^{\infty}\frac{\left(\beta \eta x^\beta\right)^{2j+1}}{\prod\limits_{m=0}^{2j+1}(\alpha+m\beta+1)}\times \\ _{p+j}F_{q+j} \Bigl(a_1, \cdot\cdot\cdot, a_p, \frac{\alpha+1}{\gamma}, \frac{\alpha+\beta+1}{\gamma}, \frac{\alpha+2\beta+1}{\gamma}, \cdot\cdot\cdot,\frac{\alpha+(2j+1)\beta+1}{\gamma}; b_1, \cdot\cdot\cdot, b_q,\\\frac{\alpha+\gamma+1}{\gamma}, \frac{\alpha+\gamma+\beta+1}{\gamma}, \frac{\alpha+\gamma+2\beta+1}{\gamma},\cdot\cdot\cdot, \frac{\alpha+\gamma+(2j+1)\beta+1}{\gamma}; \lambda x^{\gamma}\Bigr)\\ 
+\sinh\left(\eta x^\beta\right) \sum\limits_{j=0}^{\infty}\frac{\left(\beta \eta x^\beta\right)^{2j}}{\prod\limits_{m=0}^{2j}(\alpha+m\beta+1)}\times \\ _{p+j}F_{q+j} \Bigl(a_1, \cdot\cdot\cdot, a_p, \frac{\alpha+1}{\gamma}, \frac{\alpha+\beta+1}{\gamma},\frac{\alpha+2\beta+1}{\gamma}, \cdot\cdot\cdot, \frac{\alpha+2j\beta+1}{\gamma}; b_1, \cdot\cdot\cdot, b_q,\\ \frac{\alpha+\gamma+1}{\gamma},\frac{\alpha+\gamma+\beta+1}{\gamma}, \frac{\alpha+\gamma+2\beta+1}{\gamma}, \cdot\cdot\cdot, \frac{\alpha+\gamma+2j\beta+1}{\gamma}; \lambda x^{\gamma}\Bigr)\\
-\cosh\left(\eta x^\beta\right)\sum\limits_{j=0}^{\infty}\frac{\left(\beta \eta x^\beta\right)^{2j+1}}{\prod\limits_{m=0}^{2j+1}(\alpha+m\beta+1)}\times\\ _{p+j}F_{q+j} \Bigl(a_1, \cdot\cdot\cdot, a_p, \frac{\alpha+1}{\gamma}, \frac{\alpha+\beta+1}{\gamma}, \frac{\alpha+2\beta+1}{\gamma}, \cdot\cdot\cdot, \frac{\alpha+(2j+1)\beta+1}{\gamma}; b_1, \cdot\cdot\cdot, b_q,\\ \frac{\alpha+\gamma+1}{\gamma}, \frac{\alpha+\gamma+\beta+1}{\gamma}, \frac{\alpha+\gamma+2\beta+1}{\gamma},\cdot\cdot\cdot, \frac{\alpha+\gamma+(2j+1)\beta+1}{\gamma}; \lambda x^{\gamma}\Bigr). 
\label{eq2.22}
\end{multline}
\label{th3}
\end{theorem}

\begin{proof}
Using the relation  $ e^{\eta x^\beta}=\cosh\left(\eta x^\beta\right)+\sinh\left(\eta x^\beta\right)$ and  Propositions \ref{prp2} and \ref{prp3} yields
\begin{multline}
\int x^\alpha e^{\eta x^\beta}\,_pF_q (a_1, \cdot\cdot\cdot a_p; b_1, \cdot\cdot\cdot b_q; \lambda x^{\gamma})dx= \int x^\alpha\cosh\left(\eta x^\beta\right)\times \\ _pF_q (a_1, \cdot\cdot\cdot a_p; b_1, \cdot\cdot\cdot b_q; \lambda x^{\gamma})dx+\int x^\alpha\sinh\left(\eta x^\beta\right)\,_pF_q (a_1, \cdot\cdot\cdot a_p; b_1, \cdot\cdot\cdot b_q; \lambda x^{\gamma})dx\\
=x^{\alpha+1}\cosh\left(\eta x^\beta\right)\sum\limits_{j=0}^{\infty}\frac{\left(\beta \eta x^\beta\right)^{2j}}{\prod\limits_{m=0}^{2j}(\alpha+m\beta+1)}\times \\ _{p+j}F_{q+j} \Bigl(a_1, \cdot\cdot\cdot, a_p, \frac{\alpha+1}{\gamma}, \frac{\alpha+\beta+1}{\gamma},\frac{\alpha+2\beta+1}{\gamma}, \cdot\cdot\cdot, \frac{\alpha+2j\beta+1}{\gamma}\\; b_1, \cdot\cdot\cdot, b_q,\frac{\alpha+\gamma+1}{\gamma},\frac{\alpha+\gamma+\beta+1}{\gamma}, \frac{\alpha+\gamma+2\beta+1}{\gamma}, \cdot\cdot\cdot, \frac{\alpha+\gamma+2j\beta+1}{\gamma}; \lambda x^{\gamma}\Bigr)\\-x^{\alpha+1}\sinh\left(\eta x^\beta\right)\sum\limits_{j=0}^{\infty}\frac{\left(\beta \eta x^\beta\right)^{2j+1}}{\prod\limits_{m=0}^{2j+1}(\alpha+m\beta+1)} \times\\ _{p+j}F_{q+j} \Bigl(a_1, \cdot\cdot\cdot, a_p, \frac{\alpha+1}{\gamma}, \frac{\alpha+\beta+1}{\gamma}, \frac{\alpha+2\beta+1}{\gamma}, \cdot\cdot\cdot, \frac{\alpha+(2j+1)\beta+1}{\gamma};\\ b_1, \cdot\cdot\cdot, b_q,\frac{\alpha+\gamma+1}{\gamma}, \frac{\alpha+\gamma+\beta+1}{\gamma}, \frac{\alpha+\gamma+2\beta+1}{\gamma},\cdot\cdot\cdot, \frac{\alpha+\gamma+(2j+1)\beta+1}{\gamma}; \lambda x^{\gamma}\Bigr)
\\+x^{\alpha+1}\sinh\left(\eta x^\beta\right)\sum\limits_{j=0}^{\infty}\frac{\left(\beta \eta x^\beta\right)^{2j}}{\prod\limits_{m=0}^{2j}(\alpha+m\beta+1)}\times \\ _{p+j}F_{q+j} \Bigl(a_1, \cdot\cdot\cdot, a_p, \frac{\alpha+1}{\gamma}, \frac{\alpha+\beta+1}{\gamma},\frac{\alpha+2\beta+1}{\gamma}, \cdot\cdot\cdot, \frac{\alpha+2j\beta+1}{\gamma}\\; b_1, \cdot\cdot\cdot, b_q,\frac{\alpha+\gamma+1}{\gamma},\frac{\alpha+\gamma+\beta+1}{\gamma}, \frac{\alpha+\gamma+2\beta+1}{\gamma}, \cdot\cdot\cdot, \frac{\alpha+\gamma+2j\beta+1}{\gamma}; \lambda x^{\gamma}\Bigr)\\-x^{\alpha+1}\cosh\left(\eta x^\beta\right)\sum\limits_{j=0}^{\infty}\frac{\left(\beta \eta x^\beta\right)^{2j+1}}{\prod\limits_{m=0}^{2j+1}(\alpha+m\beta+1)}\times \\ _{p+j}F_{q+j} \Bigl(a_1, \cdot\cdot\cdot, a_p, \frac{\alpha+1}{\gamma}, \frac{\alpha+\beta+1}{\gamma}, \frac{\alpha+2\beta+1}{\gamma}, \cdot\cdot\cdot, \frac{\alpha+(2j+1)\beta+1}{\gamma}; b_1, \cdot\cdot\cdot, b_q,\\ \frac{\alpha+\gamma+1}{\gamma}, \frac{\alpha+\gamma+\beta+1}{\gamma}, \frac{\alpha+\gamma+2\beta+1}{\gamma},\cdot\cdot\cdot, \frac{\alpha+\gamma+(2j+1)\beta+1}{\gamma}; \lambda x^{\gamma}\Bigr)+C.
\label{eq2.23}
\end{multline}
Hence, comparing (\ref{eq2.23}) with (\ref{eq2.3}) gives (\ref{eq2.22}).
\end{proof}
\subsection{Evaluation of some non-elementary  integrals involving the trigonometric functions $\cos$ and $\sin$ and the generalized hypergeometric function $_pF_q$}
\label{subsec:2.2}

\begin{proposition} For any constants $\alpha, \beta, \eta, \lambda$ and $\gamma$ any nonzero constant ($\gamma\ne0$),
\begin{multline}
\int x^\alpha \cos\left(\eta x^\beta\right)\,_pF_q (a_1, a_2, \cdot\cdot\cdot a_p; b_1, b_2, \cdot\cdot\cdot b_q; \lambda x^{\gamma})dx\\= x^{\alpha+1}\cos\left(\eta x^\beta\right)\sum\limits_{j=0}^{\infty}\frac{(-1)^j\left(\beta \eta x^\beta\right)^{2j}}{\prod\limits_{m=0}^{2j}(\alpha+m\beta+1)}\times \\ _{p+j}F_{q+j} \Bigl(a_1, \cdot\cdot\cdot, a_p, \frac{\alpha+1}{\gamma}, \frac{\alpha+\beta+1}{\gamma},\frac{\alpha+2\beta+1}{\gamma}, \cdot\cdot\cdot, \frac{\alpha+2j\beta+1}{\gamma}\\; b_1, \cdot\cdot\cdot, b_q,\frac{\alpha+\gamma+1}{\gamma},\frac{\alpha+\gamma+\beta+1}{\gamma}, \frac{\alpha+\gamma+2\beta+1}{\gamma}, \cdot\cdot\cdot, \frac{\alpha+\gamma+2j\beta+1}{\gamma}; \lambda x^{\gamma}\Bigr)\\+x^{\alpha+1}\sin\left(\eta x^\beta\right)\sum\limits_{j=0}^{\infty}\frac{(-1)^j\left(\beta \eta x^\beta\right)^{2j+1}}{\prod\limits_{m=0}^{2j+1}(\alpha+m\beta+1)}\times\\\ _{p+j}F_{q+j} \Bigl(a_1, \cdot\cdot\cdot, a_p, \frac{\alpha+1}{\gamma}, \frac{\alpha+\beta+1}{\gamma}, \frac{\alpha+2\beta+1}{\gamma}, \cdot\cdot\cdot,\frac{\alpha+(2j+1)\beta+1}{\gamma}; b_1, \cdot\cdot\cdot, b_q,\\\frac{\alpha+\gamma+1}{\gamma}, \frac{\alpha+\gamma+\beta+1}{\gamma}, \frac{\alpha+\gamma+2\beta+1}{\gamma},\cdot\cdot\cdot, \frac{\alpha+\gamma+(2j+1)\beta+1}{\gamma}; \lambda x^{\gamma}\Bigr)+C.
\label{eq2.24}
\end{multline}
\label{prp4}
\end{proposition}
Proposition \ref{prp4}'s proof is similar to Proposition \ref{prp2}'s proof, we omit it.

\begin{proposition} For any constants $\alpha, \beta, \eta, \lambda$ and $\gamma$ any nonzero constant ($\gamma\ne0$),
\begin{multline}
\int x^\alpha \sin\left(\eta x^\beta\right)\,_pF_q (a_1, a_2, \cdot\cdot\cdot a_p; b_1, b_2, \cdot\cdot\cdot b_q; \lambda x^{\gamma})dx\\=x^{\alpha+1}\sin\left(\eta x^\beta\right)\sum\limits_{j=0}^{\infty}\frac{(-1)^j\left(\beta \eta x^\beta\right)^{2j}}{\prod\limits_{m=0}^{2j}(\alpha+m\beta+1)}\times \\ _{p+j}F_{q+j} \Bigl(a_1, \cdot\cdot\cdot, a_p, \frac{\alpha+1}{\gamma}, \frac{\alpha+\beta+1}{\gamma},\frac{\alpha+2\beta+1}{\gamma}, \cdot\cdot\cdot, \frac{\alpha+2j\beta+1}{\gamma}; b_1, \cdot\cdot\cdot, b_q,\\ \frac{\alpha+\gamma+1}{\gamma},\frac{\alpha+\gamma+\beta+1}{\gamma}, \frac{\alpha+\gamma+2\beta+1}{\gamma}, \cdot\cdot\cdot, \frac{\alpha+\gamma+2j\beta+1}{\gamma}; \lambda x^{\gamma}\Bigr)\\-x^{\alpha+1}\cos\left(\eta x^\beta\right)\sum\limits_{j=0}^{\infty}\frac{(-1)^j\left(\beta \eta x^\beta\right)^{2j+1}}{\prod\limits_{m=0}^{2j+1}(\alpha+m\beta+1)}\times \\ _{p+j}F_{q+j} \Bigl(a_1, \cdot\cdot\cdot, a_p, \frac{\alpha+1}{\gamma}, \frac{\alpha+\beta+1}{\gamma}, \frac{\alpha+2\beta+1}{\gamma}, \cdot\cdot\cdot, \frac{\alpha+(2j+1)\beta+1}{\gamma}; b_1, \cdot\cdot\cdot, b_q,\\ \frac{\alpha+\gamma+1}{\gamma}, \frac{\alpha+\gamma+\beta+1}{\gamma}, \frac{\alpha+\gamma+2\beta+1}{\gamma},\cdot\cdot\cdot, \frac{\alpha+\gamma+(2j+1)\beta+1}{\gamma}; \lambda x^{\gamma}\Bigr)+C.
\label{eq2.25}
\end{multline}
\label{prp5}
\end{proposition}
Proposition \ref{prp5}'s proof is similar to Proposition \ref{prp3}'s proof, we also omit it.

\begin{theorem} For any constants $\alpha, \beta, \eta, \lambda$ and $\gamma$ any nonzero constant ($\gamma\ne0$),
\begin{multline}
\cos\left(\eta x^\beta\right)\sum\limits_{j=0}^{\infty}\frac{(-1)^j\left(\beta \eta x^\beta\right)^{2j}}{\prod\limits_{m=0}^{2j}(\alpha+m\beta+1)}\times \\ _{p+j}F_{q+j} \Bigl(a_1, \cdot\cdot\cdot, a_p, \frac{\alpha+1}{\gamma}, \frac{\alpha+\beta+1}{\gamma},\frac{\alpha+2\beta+1}{\gamma}, \cdot\cdot\cdot, \frac{\alpha+2j\beta+1}{\gamma}\\; b_1, \cdot\cdot\cdot, b_q,\frac{\alpha+\gamma+1}{\gamma},\frac{\alpha+\gamma+\beta+1}{\gamma}, \frac{\alpha+\gamma+2\beta+1}{\gamma}, \cdot\cdot\cdot, \frac{\alpha+\gamma+2j\beta+1}{\gamma}; \lambda x^{\gamma}\Bigr)\\+\sin\left(\eta x^\beta\right)\sum\limits_{j=0}^{\infty}\frac{(-1)^j\left(\beta \eta x^\beta\right)^{2j+1}}{\prod\limits_{m=0}^{2j+1}(\alpha+m\beta+1)}\times \\ _{p+j}F_{q+j} \Bigl(a_1, \cdot\cdot\cdot, a_p, \frac{\alpha+1}{\gamma}, \frac{\alpha+\beta+1}{\gamma}, \frac{\alpha+2\beta+1}{\gamma}, \cdot\cdot\cdot,\frac{\alpha+(2j+1)\beta+1}{\gamma}; b_1,\cdot\cdot\cdot, b_q,\\ \frac{\alpha+\gamma+1}{\gamma}, \frac{\alpha+\gamma+\beta+1}{\gamma}, \frac{\alpha+\gamma+2\beta+1}{\gamma},\cdot\cdot\cdot, \frac{\alpha+\gamma+(2j+1)\beta+1}{\gamma}; \lambda x^{\gamma}\Bigr).\\
=\frac{1}{2}\Bigl[e^{i\eta x^\beta}\sum\limits_{j=0}^{\infty}\frac{\left(-i\beta \eta x^\beta\right)^j}{\prod\limits_{m=0}^{j}(\alpha+m\beta+1)}\,_{p+j}F_{q+j} \Bigl(a_1, \cdot\cdot\cdot, a_p, \frac{\alpha+1}{\gamma}, \frac{\alpha+\beta+1}{\gamma}, \cdot\cdot\cdot, \frac{\alpha+j\beta+1}{\gamma}\\; b_1, \cdot\cdot\cdot, b_q,\frac{\alpha+\gamma+1}{\gamma}, \frac{\alpha+\gamma+\beta+1}{\gamma}, \cdot\cdot\cdot, \frac{\alpha+\gamma+j\beta+1}{\gamma}; \lambda x^{\gamma}\bigr)\\
\\+e^{-i\eta x^\beta}\sum\limits_{j=0}^{\infty}\frac{\left(i\beta \eta x^\beta\right)^j}{\prod\limits_{m=0}^{j}(\alpha+m\beta+1)}\,_{p+j}F_{q+j} \Bigl(a_1, \cdot\cdot\cdot, a_p, \frac{\alpha+1}{\gamma}, \frac{\alpha+\beta+1}{\gamma}, \cdot\cdot\cdot, \frac{\alpha+j\beta+1}{\gamma}\\; b_1, \cdot\cdot\cdot, b_q,\frac{\alpha+\gamma+1}{\gamma}, \frac{\alpha+\gamma+\beta+1}{\gamma}, \cdot\cdot\cdot, \frac{\alpha+\gamma+j\beta+1}{\gamma}; \lambda x^{\gamma}\Bigr)\Bigr].
\label{eq2.26}
\end{multline}
\label{th4}
\end{theorem}

\begin{proof}
Using Euler's identity $ \cos\left(\eta x^\beta\right)= \left( e^{i\eta x^\beta}+ e^{-i\eta x^\beta}\right)/2$ and Proposition \ref{prp1} yields
\begin{multline}
\int x^\alpha\cos\left(\eta x^\beta\right)\,_pF_q (a_1, \cdot\cdot\cdot a_p; b_1, \cdot\cdot\cdot b_q; \lambda x^{\gamma})dx= \frac{1}{2}\Bigl[\\ \int x^\alpha e^{i\eta x^\beta}\,_pF_q (a_1, \cdot\cdot\cdot a_p; b_1, \cdot\cdot\cdot b_q; \lambda x^{\gamma})dx+\int x^\alpha e^{-i\eta x^\beta}\,_pF_q (a_1, \cdot\cdot\cdot a_p; b_1, \cdot\cdot\cdot b_q; \lambda x^{\gamma})dx\Bigr]\\
=\frac{1}{2}\Bigl[e^{i\eta x^\beta}\sum\limits_{j=0}^{\infty}\frac{\left(-i\beta \eta x^\beta\right)^j}{\prod\limits_{m=0}^{j}(\alpha+m\beta+1)}\,_{p+j}F_{q+j} \Bigl(a_1, \cdot\cdot\cdot, a_p, \frac{\alpha+1}{\gamma}, \frac{\alpha+\beta+1}{\gamma}, \cdot\cdot\cdot, \frac{\alpha+j\beta+1}{\gamma}\\; b_1, \cdot\cdot\cdot, b_q,\frac{\alpha+\gamma+1}{\gamma}, \frac{\alpha+\gamma+\beta+1}{\gamma}, \cdot\cdot\cdot, \frac{\alpha+\gamma+j\beta+1}{\gamma}; \lambda x^{\gamma}\bigr)\\
\\+e^{-i\eta x^\beta}\sum\limits_{j=0}^{\infty}\frac{\left(i\beta \eta x^\beta\right)^j}{\prod\limits_{m=0}^{j}(\alpha+m\beta+1)}\,_{p+j}F_{q+j} \Bigl(a_1, \cdot\cdot\cdot, a_p, \frac{\alpha+1}{\gamma}, \frac{\alpha+\beta+1}{\gamma}, \cdot\cdot\cdot, \frac{\alpha+j\beta+1}{\gamma}\\; b_1, \cdot\cdot\cdot, b_q,\frac{\alpha+\gamma+1}{\gamma}, \frac{\alpha+\gamma+\beta+1}{\gamma}, \cdot\cdot\cdot, \frac{\alpha+\gamma+j\beta+1}{\gamma}; \lambda x^{\gamma}\Bigr)\Bigr]+C.
\label{eq2.27}
\end{multline}
Hence, Comparing (\ref{eq2.27}) with (\ref{eq2.24}) gives (\ref{eq2.26}). 
\end{proof}

\begin{theorem} For any constants $\alpha, \beta, \eta, \lambda$ and $\gamma$ any nonzero constant ($\gamma\ne0$),
\begin{multline}
\sin\left(\eta x^\beta\right)\sum\limits_{j=0}^{\infty}\frac{(-1)^j\left(\beta \eta x^\beta\right)^{2j}}{\prod\limits_{m=0}^{2j}(\alpha+m\beta+1)}\times \\_{p+j}F_{q+j} \Bigl(a_1, \cdot\cdot\cdot, a_p, \frac{\alpha+1}{\gamma}, \frac{\alpha+\beta+1}{\gamma},\frac{\alpha+2\beta+1}{\gamma}, \cdot\cdot\cdot, \frac{\alpha+2j\beta+1}{\gamma}\\; b_1, \cdot\cdot\cdot, b_q,\frac{\alpha+\gamma+1}{\gamma},\frac{\alpha+\gamma+\beta+1}{\gamma}, \frac{\alpha+\gamma+2\beta+1}{\gamma}, \cdot\cdot\cdot, \frac{\alpha+\gamma+2j\beta+1}{\gamma}; \lambda x^{\gamma}\Bigr)\\-\cos\left(\eta x^\beta\right)\sum\limits_{j=0}^{\infty}\frac{(-1)^j\left(\beta \eta x^\beta\right)^{2j+1}}{\prod\limits_{m=0}^{2j+1}(\alpha+m\beta+1)}\times\\ _{p+j}F_{q+j} \Bigl(a_1, \cdot\cdot\cdot, a_p, \frac{\alpha+1}{\gamma}, \frac{\alpha+\beta+1}{\gamma}, \frac{\alpha+2\beta+1}{\gamma}, \cdot\cdot\cdot, \frac{\alpha+(2j+1)\beta+1}{\gamma}; b_1, \cdot\cdot\cdot, b_q,\\ \frac{\alpha+\gamma+1}{\gamma}, \frac{\alpha+\gamma+\beta+1}{\gamma}, \frac{\alpha+\gamma+2\beta+1}{\gamma},\cdot\cdot\cdot, \frac{\alpha+\gamma+(2j+1)\beta+1}{\gamma}; \lambda x^{\gamma}\Bigr).\\ 
=\frac{1}{2i}\Bigl[e^{i\eta x^\beta}\sum\limits_{j=0}^{\infty}\frac{\left(-i\beta \eta x^\beta\right)^j}{\prod\limits_{m=0}^{j}(\alpha+m\beta+1)}\,_{p+j}F_{q+j} \Bigl(a_1, \cdot\cdot\cdot, a_p, \frac{\alpha+1}{\gamma}, \frac{\alpha+\beta+1}{\gamma}, \cdot\cdot\cdot, \frac{\alpha+j\beta+1}{\gamma}\\; b_1, \cdot\cdot\cdot, b_q,\frac{\alpha+\gamma+1}{\gamma}, \frac{\alpha+\gamma+\beta+1}{\gamma}, \cdot\cdot\cdot, \frac{\alpha+\gamma+j\beta+1}{\gamma}; \lambda x^{\gamma}\Bigr)\\
-e^{-i\eta x^\beta}\sum\limits_{j=0}^{\infty}\frac{\left(i\beta \eta x^\beta\right)^j}{\prod\limits_{m=0}^{j}(\alpha+m\beta+1)}\,_{p+j}F_{q+j} \Bigl(a_1, \cdot\cdot\cdot, a_p, \frac{\alpha+1}{\gamma}, \frac{\alpha+\beta+1}{\gamma}, \cdot\cdot\cdot, \frac{\alpha+j\beta+1}{\gamma}\\; b_1, \cdot\cdot\cdot, b_q,\frac{\alpha+\gamma+1}{\gamma}, \frac{\alpha+\gamma+\beta+1}{\gamma}, \cdot\cdot\cdot, \frac{\alpha+\gamma+j\beta+1}{\gamma}; \lambda x^{\gamma}\Bigr)\Bigr].
\label{eq2.28}
\end{multline}
\label{th5}
\end{theorem}

\begin{proof}
Using Euler's identity $ \sin\left(\eta x^\beta\right)= \left( e^{i\eta x^\beta}- e^{-i\eta x^\beta}\right)/(2i)$ and Proposition \ref{prp1} yields
\begin{multline}
\int x^\alpha\sin\left(\eta x^\beta\right)\,_pF_q (a_1, \cdot\cdot\cdot a_p; b_1, \cdot\cdot\cdot b_q; \lambda x^{\gamma})dx= \frac{1}{2i}\Bigl[\\\int x^\alpha e^{i\eta x^\beta}\,_pF_q (a_1, \cdot\cdot\cdot a_p; b_1, \cdot\cdot\cdot b_q; \lambda x^{\gamma})dx-\int x^\alpha e^{-i\eta x^\beta}\,_pF_q (a_1, \cdot\cdot\cdot a_p; b_1, \cdot\cdot\cdot b_q; \lambda x^{\gamma})dx\Bigr]\\
=\frac{1}{2i}\Bigl[e^{i\eta x^\beta}\sum\limits_{j=0}^{\infty}\frac{\left(-i\beta \eta x^\beta\right)^j}{\prod\limits_{m=0}^{j}(\alpha+m\beta+1)}\,_{p+j}F_{q+j} \Bigl(a_1, \cdot\cdot\cdot, a_p, \frac{\alpha+1}{\gamma}, \frac{\alpha+\beta+1}{\gamma}, \cdot\cdot\cdot, \frac{\alpha+j\beta+1}{\gamma}\\; b_1, \cdot\cdot\cdot, b_q,\frac{\alpha+\gamma+1}{\gamma}, \frac{\alpha+\gamma+\beta+1}{\gamma}, \cdot\cdot\cdot, \frac{\alpha+\gamma+j\beta+1}{\gamma}; \lambda x^{\gamma}\bigr)\\
\\-e^{-i\eta x^\beta}\sum\limits_{j=0}^{\infty}\frac{\left(i\beta \eta x^\beta\right)^j}{\prod\limits_{m=0}^{j}(\alpha+m\beta+1)}\,_{p+j}F_{q+j} \Bigl(a_1, \cdot\cdot\cdot, a_p, \frac{\alpha+1}{\gamma}, \frac{\alpha+\beta+1}{\gamma}, \cdot\cdot\cdot, \frac{\alpha+j\beta+1}{\gamma}\\; b_1, \cdot\cdot\cdot, b_q,\frac{\alpha+\gamma+1}{\gamma}, \frac{\alpha+\gamma+\beta+1}{\gamma}, \cdot\cdot\cdot, \frac{\alpha+\gamma+j\beta+1}{\gamma}; \lambda x^{\gamma}\Bigr)\Bigr]+C.
\label{eq2.29}
\end{multline}
Hence, Comparing (\ref{eq2.29}) with (\ref{eq2.25}) gives (\ref{eq2.28}). 
\end{proof}

\begin{theorem} For any constants $\alpha, \beta, \eta, \lambda$ and $\gamma$ any nonzero constant ($\gamma\ne0$),
\begin{multline}
e^{i\eta x^\beta}\sum\limits_{j=0}^{\infty}\frac{\left(-i\beta \eta x^\beta\right)^j}{\prod\limits_{m=0}^{j}(\alpha+m\beta+1)}\,_{p+j}F_{q+j} \Bigl(a_1, \cdot\cdot\cdot, a_p, \frac{\alpha+1}{\gamma}, \frac{\alpha+\beta+1}{\gamma}, \cdot\cdot\cdot, \frac{\alpha+j\beta+1}{\gamma}\\; b_1, \cdot\cdot\cdot, b_q,\frac{\alpha+\gamma+1}{\gamma}, \frac{\alpha+\gamma+\beta+1}{\gamma}, \cdot\cdot\cdot, \frac{\alpha+\gamma+j\beta+1}{\gamma}; \lambda x^{\gamma}\Bigr)\\
=\cos\left(\eta x^\beta\right)\sum\limits_{j=0}^{\infty}\frac{(-1)^j\left(\beta \eta x^\beta\right)^{2j}}{\prod\limits_{m=0}^{2j}(\alpha+m\beta+1)}\times \\_{p+j}F_{q+j} \Bigl(a_1, \cdot\cdot\cdot, a_p, \frac{\alpha+1}{\gamma}, \frac{\alpha+\beta+1}{\gamma},\frac{\alpha+2\beta+1}{\gamma}, \cdot\cdot\cdot, \frac{\alpha+2j\beta+1}{\gamma}\\; b_1, \cdot\cdot\cdot, b_q,\frac{\alpha+\gamma+1}{\gamma},\frac{\alpha+\gamma+\beta+1}{\gamma}, \frac{\alpha+\gamma+2\beta+1}{\gamma}, \cdot\cdot\cdot, \frac{\alpha+\gamma+2j\beta+1}{\gamma}; \lambda x^{\gamma}\Bigr)\\+\sin\left(\eta x^\beta\right)\sum\limits_{j=0}^{\infty}\frac{(-1)^j\left(\beta \eta x^\beta\right)^{2j+1}}{\prod\limits_{m=0}^{2j+1}(\alpha+m\beta+1)}\times \\ _{p+j}F_{q+j} \Bigl(a_1, \cdot\cdot\cdot, a_p, \frac{\alpha+1}{\gamma}, \frac{\alpha+\beta+1}{\gamma}, \frac{\alpha+2\beta+1}{\gamma}, \cdot\cdot\cdot, \frac{\alpha+(2j+1)\beta+1}{\gamma};\\ b_1, \cdot\cdot\cdot, b_q,\frac{\alpha+\gamma+1}{\gamma}, \frac{\alpha+\gamma+\beta+1}{\gamma}, \frac{\alpha+\gamma+2\beta+1}{\gamma},\cdot\cdot\cdot, \frac{\alpha+\gamma+(2j+1)\beta+1}{\gamma}; \lambda x^{\gamma}\Bigr).\\ 
+i\sin\left(\eta x^\beta\right)\sum\limits_{j=0}^{\infty}\frac{(-1)^j\left(\beta \eta x^\beta\right)^{2j}}{\prod\limits_{m=0}^{2j}(\alpha+m\beta+1)}\times \\ _{p+j}F_{q+j} \Bigl(a_1, \cdot\cdot\cdot, a_p, \frac{\alpha+1}{\gamma}, \frac{\alpha+\beta+1}{\gamma},\frac{\alpha+2\beta+1}{\gamma}, \cdot\cdot\cdot, \frac{\alpha+2j\beta+1}{\gamma}\\; b_1, \cdot\cdot\cdot, b_q,\frac{\alpha+\gamma+1}{\gamma},\frac{\alpha+\gamma+\beta+1}{\gamma}, \frac{\alpha+\gamma+2\beta+1}{\gamma}, \cdot\cdot\cdot, \frac{\alpha+\gamma+2j\beta+1}{\gamma}; \lambda x^{\gamma}\Bigr)\\-i\cos\left(\eta x^\beta\right)\sum\limits_{j=0}^{\infty}\frac{(-1)^j\left(\beta \eta x^\beta\right)^{2j+1}}{\prod\limits_{m=0}^{2j+1}(\alpha+m\beta+1)}\times \\ _{p+j}F_{q+j} \Bigl(a_1, \cdot\cdot\cdot, a_p, \frac{\alpha+1}{\gamma}, \frac{\alpha+\beta+1}{\gamma}, \frac{\alpha+2\beta+1}{\gamma}, \cdot\cdot\cdot, \frac{\alpha+(2j+1)\beta+1}{\gamma}; \\ b_1, \cdot\cdot\cdot, b_q,\frac{\alpha+\gamma+1}{\gamma}, \frac{\alpha+\gamma+\beta+1}{\gamma}, \frac{\alpha+\gamma+2\beta+1}{\gamma},\cdot\cdot\cdot, \frac{\alpha+\gamma+(2j+1)\beta+1}{\gamma}; \lambda x^{\gamma}\Bigr).\\ 
\label{eq2.30}
\end{multline}
\label{th6}
\end{theorem}

\begin{proof}
Using the relation  $ e^{i\eta x^\beta}=\cos\left(\eta x^\beta\right)+i\sin\left(\eta x^\beta\right)$ and  Propositions \ref{prp4} and \ref{prp5} yields
\begin{multline}
\int x^\alpha e^{i\eta x^\beta}\,_pF_q (a_1, \cdot\cdot\cdot a_p; b_1, \cdot\cdot\cdot b_q; \lambda x^{\gamma})dx= \int x^\alpha\cos\left(\eta x^\beta\right)\times \\_pF_q (a_1, \cdot\cdot\cdot a_p; b_1, \cdot\cdot\cdot b_q; \lambda x^{\gamma})dx+i\int x^\alpha\sin\left(\eta x^\beta\right)\,_pF_q (a_1, \cdot\cdot\cdot a_p; b_1, \cdot\cdot\cdot b_q; \lambda x^{\gamma})dx\\
=x^{\alpha+1}\cos\left(\eta x^\beta\right)\sum\limits_{j=0}^{\infty}\frac{(-1)^j\left(\beta \eta x^\beta\right)^{2j}}{\prod\limits_{m=0}^{2j}(\alpha+m\beta+1)}\times \\ _{p+j}F_{q+j} \Bigl(a_1, \cdot\cdot\cdot, a_p, \frac{\alpha+1}{\gamma}, \frac{\alpha+\beta+1}{\gamma},\frac{\alpha+2\beta+1}{\gamma}, \cdot\cdot\cdot, \frac{\alpha+2j\beta+1}{\gamma}\\; b_1, \cdot\cdot\cdot, b_q,\frac{\alpha+\gamma+1}{\gamma},\frac{\alpha+\gamma+\beta+1}{\gamma}, \frac{\alpha+\gamma+2\beta+1}{\gamma}, \cdot\cdot\cdot, \frac{\alpha+\gamma+2j\beta+1}{\gamma}; \lambda x^{\gamma}\Bigr)\\+x^{\alpha+1}\sin\left(\eta x^\beta\right)\sum\limits_{j=0}^{\infty}\frac{(-1)^j\left(\beta \eta x^\beta\right)^{2j+1}}{\prod\limits_{m=0}^{2j+1}(\alpha+m\beta+1)}\times \\ _{p+j}F_{q+j} \Bigl(a_1, \cdot\cdot\cdot, a_p, \frac{\alpha+1}{\gamma}, \frac{\alpha+\beta+1}{\gamma}, \frac{\alpha+2\beta+1}{\gamma}, \cdot\cdot\cdot, \frac{\alpha+(2j+1)\beta+1}{\gamma};\\ b_1, \cdot\cdot\cdot, b_q,\frac{\alpha+\gamma+1}{\gamma}, \frac{\alpha+\gamma+\beta+1}{\gamma}, \frac{\alpha+\gamma+2\beta+1}{\gamma},\cdot\cdot\cdot, \frac{\alpha+\gamma+(2j+1)\beta+1}{\gamma}; \lambda x^{\gamma}\Bigr)
\\+i\,x^{\alpha+1}\sin\left(\eta x^\beta\right)\sum\limits_{j=0}^{\infty}\frac{(-1)^j\left(\beta \eta x^\beta\right)^{2j}}{\prod\limits_{m=0}^{2j}(\alpha+m\beta+1)}\times \\ _{p+j}F_{q+j} \Bigl(a_1, \cdot\cdot\cdot, a_p, \frac{\alpha+1}{\gamma}, \frac{\alpha+\beta+1}{\gamma},\frac{\alpha+2\beta+1}{\gamma}, \cdot\cdot\cdot, \frac{\alpha+2j\beta+1}{\gamma}\\; b_1, \cdot\cdot\cdot, b_q,\frac{\alpha+\gamma+1}{\gamma},\frac{\alpha+\gamma+\beta+1}{\gamma}, \frac{\alpha+\gamma+2\beta+1}{\gamma}, \cdot\cdot\cdot, \frac{\alpha+\gamma+2j\beta+1}{\gamma}; \lambda x^{\gamma}\Bigr)\\-i\,x^{\alpha+1}\cos\left(\eta x^\beta\right)\sum\limits_{j=0}^{\infty}\frac{(-1)^j\left(\beta \eta x^\beta\right)^{2j+1}}{\prod\limits_{m=0}^{2j+1}(\alpha+m\beta+1)}\times \\ _{p+j}F_{q+j} \Bigl(a_1, \cdot\cdot\cdot, a_p, \frac{\alpha+1}{\gamma}, \frac{\alpha+\beta+1}{\gamma}, \frac{\alpha+2\beta+1}{\gamma}, \cdot\cdot\cdot, \frac{\alpha+(2j+1)\beta+1}{\gamma};  b_1, \cdot\cdot\cdot, b_q,\\ \frac{\alpha+\gamma+1}{\gamma}, \frac{\alpha+\gamma+\beta+1}{\gamma}, \frac{\alpha+\gamma+2\beta+1}{\gamma},\cdot\cdot\cdot, \frac{\alpha+\gamma+(2j+1)\beta+1}{\gamma}; \lambda x^{\gamma}\Bigr)+C.
\label{eq2.31}
\end{multline}
Hence, comparing (\ref{eq2.31}) with (\ref{eq2.3}) (with $\eta$ replaced by $i\eta$ ) gives (\ref{eq2.30}).
\end{proof}

\section{Some applications from  applied analysis and applied science}
\label{sec:3}

Some  Fourier and Laplace integrals which have been not evaluated before are considered in this section. The integrals considered here satisfy the integrability condition (\ref{eq1.7}), and are evaluated using Lemma \ref{lem1}, Proposition \ref{prp1}, \ref{prp2} and \ref{prp3}, and the asymptotic expansions of the generalized hypergeometric function $_pF_q$.

\subsection{Evaluation of the Fourier integral $\int_{-\infty}^{+\infty} x^\alpha e^{-\theta^2 x^2} e^{ikx} dx, \, \alpha>-2$ }

The Fourier integral with $\alpha=0$,
\begin{equation}
\int\limits_{-\infty}^{+\infty}e^{-\theta^2 x^2} e^{ikx} dx=\int\limits_{-\infty}^{+\infty}e^{-\theta^2 x^2} \cos(kx) dx,
\label{eq3.1}
\end{equation}
is first evaluated using Proposition \ref{prp1}, the fundamental theorem of calculus (FTC) and the asymptotic expression of the generalized hypergeometric function $_p F_q$  for large argument. It is known, and hence there is no need to use other methods to verify the obtained results.

 The function $\cos(kx)$ is first written in terms of the hypergeometric function $ _0 F_1$. To do so, the cosine function is expressed as a Taylor series using the gamma duplication formula \cite{AB}
\begin{equation}
\Gamma(2\alpha)=(2\pi)^{-1/2} 2^{2\alpha-1/2}\Gamma(\alpha)\Gamma\left(\alpha+\frac{1}{2}\right).
\label{eq3.2}
\end{equation}
Then,
\begin{align}
\cos(kx)&=\sum\limits_{m=0}^{\infty}\frac{(-1)^m (kx)^{2m}}{(2m)!}=\sum\limits_{j=0}^{\infty}\frac{(-1)^m (kx)^{2m}}{\Gamma(2m+1)} \nonumber \\ & =\sum\limits_{m=0}^{\infty}\frac{(-1)^m (kx)^{2m}}{\Gamma\left(2\left(m+\frac{1}{2}\right)\right)}=\sum\limits_{m=0}^{\infty}\frac{(-1)^m (kx)^{2m}}{2^{2m}\left(\frac{1}{2}\right)_m  m!} \nonumber \\ & =\sum\limits_{m=0}^{\infty}\frac{\left (-\frac{k^2x^2}{4}\right)^{m}}{\left(\frac{1}{2}\right)_m  m!}=\,_0F_1\left(;\frac{1}{2};-\frac{k^2x^2}{4}\right).
\label{eq3.3}
\end{align}
Proposition \ref{prp1} gives 
\begin{multline}
\int e^{-\theta^2 x^2} e^{ikx} dx=\int e^{-\theta^2 x^2} \cos(kx) dx=\int e^{-\theta^2 x^2} \,_0F_1\left(;\frac{1}{2};-\frac{k^2x^2}{4}\right) dx\\= x  e^{-\theta^2 x^2}\sum\limits_{j=0}^{\infty}\frac{\left(2 \theta^2 x^2\right)^j}{\prod\limits_{m=0}^{j}(2m+1)}\,_{0+j}F_{1+j} \Bigl(,\frac{1}{2} ,\frac{3}{2}, \frac{5}{2},\cdot\cdot\cdot,\frac{ 2j+1}{2}\\;\frac{1}{2} ,\frac{3}{2}, \frac{5}{2}, \frac{7}{2},\cdot\cdot\cdot,\frac{ 2j+3}{2};-\frac{k^2x^2}{4}\Bigr)+C.
\label{eq3.4}
\end{multline}

The variable $x$ and the Fourier \underline{parameter} $k$ can  be separated using formula 16.10.1 in \cite{NI}. This yields 
\begin{multline}
_{0+j}F_{1+j} \Bigl(,\frac{1}{2} ,\frac{3}{2}, \frac{5}{2},\cdot\cdot\cdot,\frac{ 2j+1}{2};\frac{1}{2} ,\frac{3}{2}, \frac{5}{2}, \frac{7}{2},\cdot\cdot\cdot,\frac{ 2j+3}{2};-\frac{k^2x^2}{4}\Bigr)\\=\sum\limits_{n=0}^{\infty}\frac{(\theta_1)_n (\theta_2)_n (x^2)^n}{\left(\frac{1}{2}\right)_n(\theta_1+n)_n  n!}\, _{2}F_{2} \Bigl(\theta_1+n,\theta_2+n;\theta_3+2n+1 ,\frac{1}{2}+n;-x^2\Bigr)\times\\
_{2+j}F_{2+j} \Bigl(n,\theta_3+n,\frac{1}{2} ,\frac{3}{2}, \frac{5}{2},\cdot\cdot\cdot,\frac{ 2j+1}{2};\theta_1, \theta_2,\frac{1}{2} ,\frac{3}{2}, \frac{5}{2}, \frac{7}{2},\cdot\cdot\cdot,\frac{ 2j+3}{2};-\frac{k^2}{4}\Bigr),
\label{eq3.5}
\end{multline}
where $\theta_1, \theta_2$ and $\theta_3$ are free real or complex paramters.

Setting $\theta_1=n$ and $\theta_2=1/2$, the generalized hypergeometric function $ _{2+j}F_{2+j}$  in (\ref{eq3.5}) can be reduced to the confluent hypergeometric function $ _1F_1$.
\begin{itemize}
\item For $ j=0$, we have $$ _{2}F_{2} \Bigl(n,\theta_3+n;n ,\frac{1}{2};-\frac{k^2}{4}\Bigr)=\,_{1}F_{1} \Bigl(\theta_3+n;\frac{1}{2};-\frac{k^2}{4}\Bigr).$$
\item For $j=1$, we have $$ _{3}F_{3} \Bigl(n,\theta_3+n, \frac{1}{2};n ,\frac{1}{2}, \frac{3}{2};-\frac{k^2}{4}\Bigr)=\,_{1}F_{1} \Bigl(\theta_3+n;\frac{3}{2};-\frac{k^2}{4}\Bigr).$$
\item For $j=2$, we have $$ _{4}F_{4} \Bigl(n,\theta_3+n, \frac{1}{2}, \frac{3}{2};n ,\frac{1}{2}, \frac{3}{2}, \frac{5}{2};-\frac{k^2}{4}\Bigr)=\,_{1}F_{1} \Bigl(\theta_3+n;\frac{5}{2};-\frac{k^2}{4}\Bigr),$$
\end{itemize}
and so on. This gives 
\begin{multline}
_{2+j}F_{2+j} \Bigl(n,\theta_3+n,\frac{1}{2} ,\frac{3}{2}, \frac{5}{2},\cdot\cdot\cdot,\frac{ 2j+1}{2};\theta_1, \theta_2,\frac{1}{2} ,\frac{3}{2}, \frac{5}{2}, \frac{7}{2},\cdot\cdot\cdot,\frac{ 2j+3}{2};-\frac{k^2}{4}\Bigr)\\ =\,_{1}F_{1} \Bigl(\theta_3+n;j+\frac{1}{2};-\frac{k^2}{4}\Bigr).
\label{eq3.6}
\end{multline}
Moreover,
\begin{equation}
_{2}F_{2} \Bigl(\theta_1+n,\theta_2+n;\theta_3+2n+1 ,\frac{1}{2}+n;-x^2\Bigr)=\,_{1}F_{1} \Bigl(2n;\theta_3+2n+1 ;-x^2\Bigr).
\label{eq3.7}
\end{equation}

Using Pochhammer' s notation and/or the gamma duplication formula, it is straightforward to obtain that  $(\theta_1)_n=(n)_n=(1/2\sqrt{\pi})\Gamma\left(n+{1}/{2}\right)$ and $ (\theta_3+n)_n=\Gamma(\theta_3+2n)/\Gamma(\theta_3+n)$. Then after rearranging terms, we have 
\begin{multline}
_{0+j}F_{1+j} \Bigl(,\frac{1}{2} ,\frac{3}{2}, \frac{5}{2},\cdot\cdot\cdot,\frac{ 2j+1}{2};\frac{1}{2} ,\frac{3}{2}, \frac{5}{2}, \frac{7}{2},\cdot\cdot\cdot,\frac{ 2j+3}{2};-\frac{k^2x^2}{4}\Bigr)\\=\frac{1}{2\sqrt{\pi}}\sum\limits_{n=0}^{\infty}\frac{\Gamma\left(n+\frac{1}{2}\right)\Gamma(\theta_3+n)}{\Gamma(\theta_3+2n) }\frac{ (4x^2)^n}{ n!}\, _{1}F_{1} \Bigl(2n;\theta_3+2n+1;-x^2\Bigr)\times\\
_{1}F_{1} \Bigl(\theta_3+n;j+\frac{ 1}{2};-\frac{k^2}{4}\Bigr).
\label{eq3.8}
\end{multline}

Let us now consider that $x$ is large ($|x|\gg1$). In that case, one can use formula 13.1.5 in \cite{AB} and obtain 
\begin{equation}
_{1}F_{1} \Bigl(2n;\theta_3+2n+1;-x^2\Bigr)\sim \frac{\Gamma(\theta_3+2n+1)}{\Gamma(\theta_3+1) } (x^2)^{-2n}, |x|\gg1.
\label{eq3.9}
\end{equation}
Substituting (\ref{eq2.9}) in (\ref{eq2.8}) yields
\begin{multline}
_{0+j}F_{1+j} \Bigl(,\frac{1}{2} ,\frac{3}{2}, \frac{5}{2},\cdot\cdot\cdot,\frac{ 2j+1}{2};\frac{1}{2} ,\frac{3}{2}, \frac{5}{2}, \frac{7}{2},\cdot\cdot\cdot,\frac{ 2j+3}{2};-\frac{k^2x^2}{4}\Bigr)=\\ \frac{1}{2\sqrt{\pi}}\sum\limits_{n=0}^{\infty}\frac{\Gamma\left(n+\frac{1}{2}\right)\Gamma(\theta_3+n)\Gamma(\theta_3+2n+1)}{\Gamma(\theta_3+2n) \Gamma(\theta_3+1)}\frac{ \left(\frac{4}{x^2}\right)^n}{ n!}\, _{1}F_{1} \Bigl(\theta_3+n;j+\frac{ 1}{2};-\frac{k^2}{4}\Bigr), |x|\gg1.
\label{eq3.10}
\end{multline}
We now note that if $x$ becomes large ($x\to\pm\infty$), then all the terms in (\ref{eq3.10}) will vanish except the first term corresponding to $n=0$. This yields
\begin{multline}
_{0+j}F_{1+j} \Bigl(,\frac{1}{2} ,\frac{3}{2}, \frac{5}{2},\cdot\cdot\cdot,\frac{ 2j+1}{2};\frac{1}{2} ,\frac{3}{2}, \frac{5}{2}, \frac{7}{2},\cdot\cdot\cdot,\frac{ 2j+3}{2};-\frac{k^2x^2}{4}\Bigr)=\\ \frac{1}{2}\, _{1}F_{1} \Bigl(\theta_3;j+\frac{ 1}{2};-\frac{k^2}{4}\Bigr)=\frac{1}{2}\sum\limits_{l=0}^{\infty}\frac{(\theta_3)_l}{\left(j+\frac{ 1}{2}\right)_l} \frac{\left(-\frac{k^2}{4}\right)^l}{ l!}, |x|\gg1.
\label{eq3.11}
\end{multline}

Next, substituting (\ref{eq3.11}) in (\ref{eq2.4}), we obtain 
\begin{multline}
\int e^{-\theta^2 x^2} e^{ikx} dx=\int e^{-\theta^2 x^2} \,_0F_1\left(;\frac{1}{2};-\frac{k^2x^2}{4}\right) dx
\\ = \frac{1}{2} x  e^{-\theta^2 x^2}\sum\limits_{j=0}^{\infty}\frac{\left(2 \theta^2 x^2\right)^j}{\prod\limits_{m=0}^{j}(2m+1)}\sum\limits_{l=0}^{\infty}\frac{(\theta_3)_l}{\left(j+\frac{ 1}{2}\right)_l} \frac{\left(-\frac{k^2}{4}\right)^l}{ l!} +C, |x|\gg1.
\label{eq3.12}
\end{multline}

Observe now that 
$$\left(j+\frac{ 1}{2}\right)_l=\frac{\Gamma\left(j+\frac{ 1}{2}+l \right)}{\Gamma\left(j+\frac{ 1}{2}\right)}=\frac{\left(l+\frac{ 1}{2} \right)_j\Gamma\left(l+\frac{ 1}{2} \right)}{\left(\frac{ 1}{2}\right)_j\Gamma\left(\frac{ 1}{2}\right)}\,\,\mbox{and}\,\,\prod\limits_{m=0}^{j}(2m+1) =2^j \frac{ 1}{2}\left(\frac{3}{2}\right)_j.$$
Substituting into (\ref{eq3.12}) and rearranging terms yields

\begin{align}
&\int e^{-\theta^2 x^2} e^{ikx} dx=\int e^{-\theta^2 x^2} \,_0F_1\left(;\frac{1}{2};-\frac{k^2x^2}{4}\right) dx
\nonumber \\ &= \frac{1}{2} x  e^{-\theta^2 x^2}\sum\limits_{j=0}^{\infty}\frac{\left(2 \theta^2 x^2\right)^j}{\prod\limits_{m=0}^{j}(2m+1)}\sum\limits_{l=0}^{\infty}\frac{(\theta_3)_l}{\left(j+\frac{ 1}{2}\right)_l} \frac{\left(-\frac{k^2}{4}\right)^l}{ l!} +C
\nonumber \\ &= \Gamma\left(\frac{1}{2}\right) x  e^{-\theta^2 x^2}\sum\limits_{l=0}^{\infty}\frac{(\theta_3)_l}{\Gamma\left(l+\frac{ 1}{2} \right)} \frac{\left(-\frac{k^2}{4}\right)^l}{ l!} \sum\limits_{j=0}^{\infty}\frac{\left(\frac{1}{2}\right)_j (1)_j}{\left(\frac{3}{2}\right)_j \left(l+\frac{ 1}{2} \right)_j}\frac{\left(\theta^2 x^2\right)^j}{j!}+C
\nonumber \\ &= \Gamma\left(\frac{1}{2}\right) x  e^{-\theta^2 x^2}\sum\limits_{l=0}^{\infty}\frac{(\theta_3)_l}{\Gamma\left(l+\frac{ 1}{2} \right)} \frac{\left(-\frac{k^2}{4}\right)^l}{ l!} \, _{2}F_{2} \Bigl(\frac{1}{2},1;\frac{3}{2},l+ \frac{1}{2};\theta^2 x^2\Bigr)+C, |x|\gg1.
\label{eq3.13}
\end{align}

It can readily be shown  using formula 5.25 in \cite{N2} (one can also use formulas 16.11.1, 16.11.2 and 16.11.7 in \cite{NI}) that 
\begin{equation}
_{2}F_{2} \Bigl(\frac{1}{2},1;\frac{3}{2},l+ \frac{1}{2};\theta^2 x^2\Bigr)\sim\frac{\Gamma\left(\frac{ 3}{2} \right)\Gamma\left(l+\frac{ 1}{2} \right)}{\Gamma\left(\frac{ 1}{2} \right)} (\theta^2 x^2)^{-l-1/2} e^{\theta^2 x^2},  |x|\gg1
\label{eq3.14}
\end{equation}

Substituting (\ref{eq3.14}) in (\ref{eq3.13}) and rearranging terms gives
\begin{align}
\int e^{-\theta^2 x^2} e^{ikx} dx&=\int e^{-\theta^2 x^2} \cos(kx) dx
\nonumber \\ &=\Gamma\left(\frac{3}{2}\right) \frac{ x}{|\theta x|}\sum\limits_{l=0}^{\infty}\frac{(\theta_3)_l}{(x^2)^l} \frac{\left(-\frac{k^2}{4\theta^2}\right)^l}{ l!}, |x|\gg1.
\label{eq3.15}
\end{align}
 
Applying the fundamental theorem of calculus ( FTC) gives
\begin{align}
\int\limits_{-L}^{+L} e^{-\theta^2 x^2} e^{ikx} dx&=\int\limits_{-L}^{+L} e^{-\theta^2 x^2} \cos(kx) dx
\nonumber \\ &=\frac{2\Gamma\left(\frac{3}{2}\right)}{|\theta|}\sum\limits_{l=0}^{\infty}\frac{(\theta_3)_l}{(L^2)^l} \frac{\left(-\frac{k^2}{4\theta^2}\right)^l}{ l!}, L\gg1.
\nonumber \\ &=\frac{2\Gamma\left(\frac{3}{2}\right)}{|\theta|}\sum\limits_{l=0}^{\infty}\frac{(L^2)^l}{(L^2)^l} \frac{\left(-\frac{k^2}{4\theta^2}\right)^l}{ l!}=\frac{2\Gamma\left(\frac{3}{2}\right)}{|\theta|}\sum\limits_{l=0}^{\infty} \frac{\left(-\frac{k^2}{4\theta^2}\right)^l}{ l!}
\label{eq3.16}
\end{align}
 since $\theta_3$ is a real or a complex free parameter, and if $\mbox{Re}(\theta_3) \gg 1$, then $(\theta_3)_l=\prod_{m=1}^{l} (\theta_3+m-1) \to  (\theta_3)^l$.
 Hence,
\begin{align}
\int\limits_{-\infty}^{+\infty}e^{-\theta^2 x^2} e^{ikx} dx&=\int\limits_{-\infty}^{+\infty} e^{-\theta^2 x^2} \cos(kx) dx=\int\limits_{-\infty}^{+\infty} e^{-\theta^2 x^2} \,_0F_1\left(;\frac{1}{2};-\frac{k^2x^2}{4}\right) dx
\nonumber \\ &=\frac{2\Gamma\left(\frac{3}{2}\right)}{|\theta|}\sum\limits_{l=0}^{\infty} \frac{\left(-\frac{k^2}{4\theta^2}\right)^l}{ l!}=\frac{\sqrt{\pi}}{|\theta|}e^{-\frac{k^2}{4\theta^2}}
\label{eq3.17}
\end{align}
as expected.

The related Fourier integral $\int_{-\infty}^{+\infty} x^\alpha  e^{-\theta^2 x^2} e^{ikx} dx$ can also  be evaluated. 
\begin{theorem} The Fourier integral 
\begin{multline}
\int\limits_{-\infty}^{+\infty} x^\alpha  e^{-\theta^2 x^2} e^{ikx} dx= \begin{cases} \frac{\Gamma\left(\frac{\alpha +1}{2}\right)}{|\theta|^{\alpha+1}}\, _{1}F_{1} \Bigl(\frac{\alpha +1}{2};\frac{1}{2}; -\frac{k^2}{4\theta^2}\Bigr) , & \text{if}\, \alpha>-1 \,\text{ and even}.  \\
i \frac{\Gamma\left(\frac{\alpha}{2}+1\right)}{\theta^2|\theta|^\alpha}k\, _{1}F_{1} \Bigl(\frac{\alpha}{2}+1;\frac{3}{2}; -\frac{k^2}{4\theta^2}\Bigr) , & \text{if}\, \alpha>-2 \,\text{ and odd}. \\
\end{cases}
\label{eq3.18}
\end{multline}
\label{th7}
\end{theorem}
Using the same procedure as in the evaluation of $\int_{-\infty}^{+\infty}  e^{-\theta^2 x^2} e^{ikx} dx$, it can be shown that (\ref{eq3.18}) holds. So the proof of Theorem \ref{th7} is omitted. It can readily be verified that if $\alpha=0$ and noticing that $\alpha=0$ is even, then Theorem \ref{th7}  gives  (\ref{eq3.17}). Moreover, if  $\alpha=1$ and noticing that $\alpha=1$ is odd, then Theorem \ref{th7}  gives
\begin{multline}\int\limits_{-\infty}^{+\infty} x  e^{-\theta^2 x^2} e^{ikx} dx=\frac{ik}{2\theta^2}\frac{\sqrt{\pi}}{|\theta|}e^{-\frac{k^2}{4\theta^2}}\\ =\frac{ik}{2\theta^2}\int\limits_{-\infty}^{+\infty}e^{-\theta^2 x^2} e^{ikx} dx=-\frac{1}{2\theta^2}\int\limits_{-\infty}^{+\infty}\left(\frac{d}{dx}e^{-\theta^2 x^2}\right) e^{ikx} dx.
\label{eq3.19}
\end{multline}

\subsection{Evaluation of the Laplace integral $\int_{0}^{+\infty}x^\alpha e^{-\theta^2 x^2} e^{-ux} dx$}

\begin{theorem} The Laplace integral 
\begin{equation}
\int\limits_{0}^{+\infty}x^\alpha e^{-\theta^2 x^2} e^{-ux} dx=\frac{\Gamma(\alpha+1)}{u^{\alpha+1}}\,_{2}F_{0}\left(\frac{\alpha+1}{2},\frac{\alpha}{2}+1;\,;-\frac{4\theta^2}{u^2}\right),\,\alpha>-1.
\label{eq3.20}
\end{equation}
\label{th8}
\end{theorem}

\begin{proof}

Proposition \ref{prp1} gives

\begin{multline}
\int x^\alpha e^{-\theta^2 x^2} e^{-ux} dx=\int x^\alpha \,_0F_0 \left(;\,;-\theta^2 x^2\right) e^{-ux} dx\\
x^{\alpha+1} e^{-ux}\sum\limits_{j=0}^{\infty}\frac{(u x)^j}{\prod\limits_{m=0}^{j}(\alpha+m+1)}\,_{0+j}F_{0+j} \Bigl(, \frac{\alpha+1}{2}, \frac{\alpha+2}{2}, \frac{\alpha+3}{2},\cdot\cdot\cdot, \frac{\alpha+j+2}{\gamma}\\; ,\frac{\alpha+3}{2}, \frac{\alpha+4}{2}, \frac{\alpha+5}{2},\cdot\cdot\cdot, \frac{\alpha+j+3}{\gamma}; \lambda x^{\gamma}\Bigr)+C.
\label{eq3.21}
\end{multline}

Then using Lemma \ref{lem1} yields

\begin{multline}
\int x^\alpha e^{-\theta^2 x^2} e^{-ux} dx=\int x^\alpha \,_0F_0 \left(;\,;-\theta^2 x^2\right) e^{-ux} dx\\
=x^{\alpha+1} e^{-ux}\sum\limits_{n=0}^{\infty}\frac{(-\theta^2 x^2)^n}{n!}\sum\limits_{j=0}^{\infty}\frac{(u x)^j}{\prod\limits_{m=1}^{j}(2n+\alpha+m+1)}+C
\\=x^{\alpha+1} e^{-ux}\sum\limits_{n=0}^{\infty}\frac{(-\theta^2 x^2)^n}{(2n+\alpha+1)n!}\sum\limits_{j=0}^{\infty}\frac{(u x)^j}{\prod\limits_{m=1}^{j}(2n+\alpha+m+1)}+C
\\=x^{\alpha+1} e^{-ux}\sum\limits_{n=0}^{\infty}\frac{(-\theta^2 x^2)^n}{(2n+\alpha+1)n!}\sum\limits_{j=0}^{\infty}\frac{(1)_j(u x)^j}{(2n+\alpha+2)_j j!}+C
\\=x^{\alpha+1} e^{-ux}\sum\limits_{n=0}^{\infty}\frac{(-\theta^2 x^2)^n}{(2n+\alpha+1)n!}\,_1F_1 \left(1;2n+\alpha+1;u x\right) +C.
\label{eq3.22}
\end{multline}
Applying the FTC yields
\begin{multline}
\int\limits_{0}^{+\infty} x^\alpha e^{-\theta^2 x^2} e^{-ux} dx=\int\limits_{0}^{+\infty} x^\alpha \,_0F_0 \left(;\,;-\theta^2 x^2\right) e^{-ux} dx\\
=\lim\limits_{x\to+\infty}x^{\alpha+1} e^{-ux}\sum\limits_{n=0}^{\infty}\frac{(-\theta^2 x^2)^n}{(2n+\alpha+1)n!}\,_1F_1 \left(1;2n+\alpha+2;u x\right).
\label{eq3.23}
\end{multline}
Formula 13.1.5 in \cite{AB} gives the asymptotic expression 
\begin{equation}
_{1}F_{1}\left(1;2n+\alpha+2;u x\right)\sim \frac{\Gamma(2n+\alpha+2)}{ (ux)^{2n+\alpha+1}} e^{ux}, |x|\gg1.
\label{eq3.24}
\end{equation}
Substituting (\ref{eq3.24}) in (\ref{eq3.23}), using Pochhammer notation and the gamma duplication formula (\ref{eq3.2}) yields
\begin{multline}
\int\limits_{0}^{+\infty} x^\alpha e^{-\theta^2 x^2} e^{-ux} dx=\int\limits_{0}^{+\infty} x^\alpha \,_0F_0 \left(;\,;-\theta^2 x^2\right) e^{-ux} dx\\
=\frac{1}{u^{\alpha+1}}\sum\limits_{n=0}^{\infty}\frac{\Gamma(2n+\alpha+2)\left(-\frac{\theta^2}{u^2}\right)^n}{(2n+\alpha+1)n!}=\frac{1}{u^{\alpha+1}}\sum\limits_{n=0}^{\infty}\Gamma(2n+\alpha+1)\frac{\left(-\frac{\theta^2}{u^2}\right)^n}{n!}\\
=\frac{\Gamma(\alpha+1)}{u^{\alpha+1}}\sum\limits_{n=0}^{\infty}\left(\frac{\alpha+1}{2}\right)_n\left(\frac{\alpha}{2}+1\right)_n\frac{\left(-\frac{4\theta^2}{u^2}\right)^n}{n!},\,\alpha>-1\\=
\frac{\Gamma(\alpha+1)}{u^{\alpha+1}}\,_{2}F_{0}\left(\frac{\alpha+1}{2},\frac{\alpha}{2}+1;\,;-\frac{4\theta^2}{u^2}\right),\,\alpha>-1.
\label{eq3.25}
\end{multline}
A simple way to verify that  (\ref{eq3.25}) is correct is to  expand the Gaussian in terms of its Taylor series and to use the  linear property of Laplace integrals. Then,

\begin{multline}
\int\limits_{0}^{+\infty} x^\alpha e^{-\theta^2 x^2} e^{-ux} dx=\int\limits_{0}^{+\infty} x^\alpha\sum\limits_{n=0}^{\infty} \frac{\left(-\theta^2 x^2\right)^n}{n!} e^{-ux} dx\\
=\sum\limits_{n=0}^{\infty} \frac{\left(-\theta^2\right)^{n}}{n!}\int\limits_{0}^{+\infty}  x^{2n+\alpha} e^{-ux} dx=\sum\limits_{n=0}^{\infty} \frac{\left(-\theta^2\right)^{n}}{n!}\frac{\Gamma(2n+\alpha+1)}{ u^{2n+\alpha+1}}
\\=\frac{1}{u^{\alpha+1}}\sum\limits_{n=0}^{\infty}\Gamma(2n+\alpha+1)\frac{\left(-\frac{\theta^2}{u^2}\right)^n}{n!},
\label{eq3.26}
\end{multline}
which is exactly (\ref{eq3.25}).
\end{proof}

\begin{corollary} The Laplace integral (or transform) of the Gaussian is 
\begin{equation}
\int\limits_{0}^{+\infty}e^{-\theta^2 x^2} e^{-ux} dx=\frac{1}{u}\,_{2}F_{0}\left(\frac{1}{2},1;\,;-\frac{4\theta^2}{u^2}\right).
\label{eq3.27}
\end{equation}
Therefore, that of the error function, $\text{erf} (x)=\int_{0}^{x} e^{-v^2} dv$,  is given by 
\begin{equation}
\int\limits_{0}^{+\infty}\text{erf} (x) e^{-ux} dx=\frac{1}{u}\int\limits_{0}^{+\infty}e^{- x^2} e^{-ux} dx=\frac{1}{u^2}\,_{2}F_{0}\left(\frac{1}{2},1;\,;-\frac{4}{u^2}\right).
\label{eq3.28}
\end{equation}
\label{cor1}
\end{corollary}
The results  in Corollary \ref{cor1} are obtained by setting $\alpha=0$ in Theorem \ref{th7}, while $\theta$ is set to one ($\theta=1$) in (\ref{eq3.27}) in order to obtain (\ref{eq3.28}).

\subsection{Solution of the Orr-Sommerfeld equation with the plane Couette mean flow background in the boundary layer}
In the short-wave limit approximation, the Orr-Sommerfeld equation with the plane Couette flow background  $\bar{U}(y)=y, 0\le y<\infty$, \cite{SH}
\begin{multline}
\phi_{yyyy}-\left[2r^2k^2+i \,\text{Re}\, k\left(\bar{U}(y)-\frac{\omega}{k}\right)\right]\phi_{yy}\\+\left[r^4k^4+i \,r^2\text{Re}\, k^3\left(\bar{U}(y)-\frac{\omega}{k}\right)+i \,\text{Re}\, k\bar{U}_{yy}(y)\right]\phi=0,
\label{eq3.29}
\end{multline}
where $k$ is the wavenumber, $r\gg1$ is the space aspect ratio, $\mbox{Re}$ is the Reynolds number, $\omega$ is the intrinsic wave frequency and the subscript $y$ stands for differentiation with respect to $y$, has the solution \cite{N5} 
\begin{multline}
\phi(y)=\frac{e^{-rk y}}{rk}\int\limits_0^y \cosh(rk\xi)\mbox{Ai}[ (i \,\mbox{Re}\,k)^{1/3} (\xi-\lambda)]d\xi\\+
\frac{\cosh(rky)}{rk}\int\limits_y^\infty e^{-rk\xi}\mbox{Ai}[ (i\, \mbox{Re}\, k)^{1/3} (\xi-\lambda)]d\xi,
\label{eq3.30}
\end{multline}
where $\mbox{Ai}$ is the Airy function and $\lambda=i \,\text{Re}\, \omega-r^2k^2$.
 
This solution was derived using Green's function methods. An outer solution (see for example \cite{BO}) which is valid for $|y|\gg1$, was obtained as well and is given by equation (68) in \cite{N5}. Here, we rather obtain a solution valid for all $y\ge0$ by evaluating the integrals in (\ref{eq3.18}) using Propositions \ref{prp1}, \ref{prp2} and \ref{prp3}, and write $\phi(y)$ in terms of some  series involving the hypergeometric function $ _{2}F_{3}$ with well known mathematical properties (e.g. behavior for $|y|\gg1$). In that case we can use the asymptotic expansion of the hypergeometric function $ _{2}F_{3}$ when $y$ goes to infinity ($y\to+\infty$), which is given by formula 3.18 in \cite{N2}. 

To do so,  the Airy function $\mbox{Ai}$ is first written  in terms of the hypergeometric function $ _{0}F_{1}$ \cite{NI} as 
\begin{multline}
\mbox{Ai}[ (i \,\mbox{Re}\,k)^{1/3} (\xi-\lambda)]=\frac{_{0}F_{1}\left(;\frac{2}{3};\frac{(i\, \mbox{Re}\, k) (\xi-\lambda)^3}{9}\right)}{\Gamma\left(\frac{2}{3}\right)}\\- \frac{3^{\frac{1}{3}}(\xi-\lambda)}{\Gamma\left(\frac{1}{3}\right)}\,_{0}F_{1}\left(;\frac{4}{3};\frac{(i\, \mbox{Re}\, k) (\xi-\lambda)^3}{9}\right).
\label{eq3.31}
\end{multline}
Substituting in (\ref{eq3.30}) yields
\begin{multline}
\phi(y)\sim\frac{e^{-rk y}}{3^{\frac{2}{3}}\, rk}\times
\\ \int\limits_0^y \cosh(rk\xi)\left[\frac{_{0}F_{1}\left(;\frac{2}{3};\frac{(i\, \mbox{Re}\, k) (\xi-\lambda)^3}{9}\right)}{\Gamma\left(\frac{2}{3}\right)}- \frac{3^{\frac{1}{3}}(\xi-\lambda)}{\Gamma\left(\frac{1}{3}\right)}\,_{0}F_{1}\left(;\frac{4}{3};\frac{(i\, \mbox{Re}\, k) (\xi-\lambda)^3}{9}\right)\right]d\xi\\+
\frac{\cosh(rky)}{3^{\frac{2}{3}}\, rk}\times
\\ \int\limits_y^\infty e^{-rk\xi}\left[\frac{_{0}F_{1}\left(;\frac{2}{3};\frac{(i\, \mbox{Re}\, k) (\xi-\lambda)^3}{9}\right)}{\Gamma\left(\frac{2}{3}\right)}- \frac{3^{\frac{1}{3}}(\xi-\lambda)}{\Gamma\left(\frac{1}{3}\right)}\,_{0}F_{1}\left(;\frac{4}{3};\frac{(i\, \mbox{Re}\, k) (\xi-\lambda)^3}{9}\right)\right]d\xi.
\label{eq3.32}
\end{multline}
 Hence, making use of Propositions \ref{prp1}, \ref{prp2} and \ref{prp3} and rearranging terms give
\begin{multline}
\phi(y)\sim\frac{e^{-rk y}}{3^{\frac{2}{3}}\, rk}\Bigl\{\frac{\lambda}{\Gamma\left(\frac{2}{3}\right)}\Bigl[\Bigl(\cosh^2(rk\lambda)+\sinh^2(rk\lambda)\Bigr) \sum\limits_{j=0}^{\infty}\frac{(r^2k^2\lambda^2)^j}{\Gamma\left(2j+2\right)}
\\-2rk\lambda \cosh(rk\lambda)\sinh(rk\lambda) \sum\limits_{j=0}^{\infty}\frac{(r^2k^2\lambda^2)^j}{\Gamma\left(2j+3\right)}\Bigr]\\ 
\times \, _{2}F_{3} \Bigl(\frac{1}{3},1;\frac{j+1}{3},\frac{j+2}{3}, \frac{j+3}{3};-\frac{(i\, \mbox{Re}\, k) \lambda^3}{9}\Bigr)\\
  \Bigl[ \frac{\cosh(rk\lambda)}{\Gamma\left(\frac{2}{3}\right)}\Bigl[(y-\lambda)\cosh(rk(y-\lambda)) \sum\limits_{j=0}^{\infty}\frac{(r^2k^2(y-\lambda)^2)^j}{\Gamma\left(2j+2\right)}
\\-rk(y-\lambda)^2\sinh(rk(y-\lambda)) \sum\limits_{j=0}^{\infty}\frac{(r^2k^2(y-\lambda)^2)^j}{\Gamma\left(2j+3\right)}\Bigr]\\ -\frac{\sinh(rk\lambda)}{\Gamma\left(\frac{2}{3}\right)}\Bigl[(y-\lambda)\sinh(rk(y-\lambda)) \sum\limits_{j=0}^{\infty}\frac{(r^2k^2(y-\lambda)^2)^j}{\Gamma\left(2j+2\right)}
\\-rk(y-\lambda)^2\cosh(rk(y-\lambda)) \sum\limits_{j=0}^{\infty}\frac{(r^2k^2(y-\lambda)^2)^j}{\Gamma\left(2j+3\right)}\Bigr]\Bigr]
\\ \times \, _{2}F_{3} \Bigl(\frac{1}{3},1;\frac{j+1}{3},\frac{j+2}{3}, \frac{j+3}{3};\frac{(i\, \mbox{Re}\, k) (y-\lambda)^3}{9}\Bigr)\\
-\frac{\lambda^2}{\Gamma\left(\frac{2}{3}\right)}\Bigl[\Bigl(\cosh^2(rk\lambda)+\sinh^2(rk\lambda)\Bigr) \sum\limits_{j=0}^{\infty}\frac{(r^2k^2\lambda^2)^j}{\Gamma\left(2j+3\right)}
\\-2rk\lambda \cosh(rk\lambda)\sinh(rk\lambda) \sum\limits_{j=0}^{\infty}\frac{(r^2k^2\lambda^2)^j}{\Gamma\left(2j+4\right)}\Bigr]\\ 
\times \, _{2}F_{3} \Bigl(\frac{1}{3},1;\frac{j+2}{3},\frac{j+3}{3}, \frac{j+4}{3};-\frac{(i\, \mbox{Re}\, k) \lambda^3}{9}\Bigr)
\\
\Bigl[\frac{3^{\frac{1}{3}}\cosh(rk\lambda)}{\Gamma\left(\frac{1}{3}\right)}\Bigl[(y-\lambda)^2\cosh(rk(y-\lambda)) \sum\limits_{j=0}^{\infty}\frac{(r^2k^2(y-\lambda)^2)^j}{\Gamma\left(2j+3\right)}
\\-rk(y-\lambda)^3\sinh(rk(y-\lambda)) \sum\limits_{j=0}^{\infty}\frac{(r^2k^2(y-\lambda)^2)^j}{\Gamma\left(2j+4\right)}\Bigr]\\ -\frac{3^{\frac{1}{3}}\sinh(rk\lambda)}{\Gamma\left(\frac{1}{3}\right)}\Bigl[(y-\lambda)^2\sinh(rk(y-\lambda)) \sum\limits_{j=0}^{\infty}\frac{(r^2k^2(y-\lambda)^2)^j}{\Gamma\left(2j+3\right)}
\\-rk(y-\lambda)^3\cosh(rk(y-\lambda)) \sum\limits_{j=0}^{\infty}\frac{(r^2k^2(y-\lambda)^2)^j}{\Gamma\left(2j+4\right)}\Bigr]\Bigr]
\\ \times \, _{2}F_{3} \Bigl(\frac{2}{3},1;\frac{j+2}{3},\frac{j+3}{3}, \frac{j+4}{3};\frac{(i\, \mbox{Re}\, k) (y-\lambda)^3}{9}\Bigr)\Bigr\}
\\+\frac{\cosh(rky)e^{-rk y}}{3^{\frac{2}{3}}\, rk}\Bigl\{
\\ \frac{(y-\lambda)}{\Gamma\left(\frac{2}{3}\right)} \sum\limits_{j=0}^{\infty}\frac{(rk(y-\lambda))^j}{\Gamma\left(j+2\right)} \, _{2}F_{3} \Bigl(\frac{1}{3},1;\frac{j+1}{3},\frac{j+2}{3}, \frac{j+3}{3};\frac{(i\, \mbox{Re}\, k) (y-\lambda)^3}{9}\Bigr)
\\+ \frac{3^{\frac{1}{3}}(y-\lambda)^2}{\Gamma\left(\frac{1}{3}\right)} \sum\limits_{j=0}^{\infty}\frac{(rk(y-\lambda))^j}{\Gamma\left(j+3\right)} \, _{2}F_{3} \Bigl(\frac{2}{3},1;\frac{j+2}{3},\frac{j+3}{3}, \frac{j+4}{3};\frac{(i\, \mbox{Re}\, k) (y-\lambda)^3}{9}\Bigr)\Bigr\}.
\label{eq3.32}
\end{multline}

\section{Concluding remarks and discussion}\label{sec:4}
Some integrals involving the generalized hypergeometric functions $_pF_q$ were evaluated in terns of infinite series involving $_pF_q$  in section \ref{sec:2} (Propositions \ref{prp1}-\ref{prp5}). If $_pF_q$  is entire on $\mathbb{C}$, then the behaviors of  $_pF_q$  for large argument, e.g. $|x|\gg1$, can be assessed, see  \cite{NI} section 16.10. In that case, the integrals evaluated in this study may find applications. Some applications in applied analysis were given in section \ref{sec:3}. For example, the Fourier integral $\int_{-\infty}^{+\infty}x^\alpha e^{-\theta^2 x^2} e^{ikx},\,\alpha>-2$,  and the Laplace integral  $\int_{0}^{+\infty}x^\alpha e^{-\theta^2 x^2} e^{-ux},\,\alpha>-1$, were evaluated by first writing them  in terms of infinite series involving the generalized hypergeometric function  $_pF_q$ (Theorems \ref{th7} and \ref{th8}). Although complicated, the method developed in this work can be used to evaluate other useful integrals which have not been evaluated in the past. In addition, using hyperbolic and Euler identities, some identities involving infinite series of the generalized hypergeometric function $_pF_q$ were obtained  in section \ref{sec:2} ( Theorems \ref{th1}-\ref{th6}).


\begin{thebibliography}{9}

%
%
%
%

\bibitem{AB}  M. Abramowitz, I.A. Stegun,
\emph{Handbook of mathematical functions with formulas, graphs and mathematical tables}.
{Nat. Bur. Stands.} 1964. 

\bibitem{AQ}{P. Agarwal, F. Qi, M. Chand, S. Jain}, Certain  integrals  involving the generalized hypergeometric function and the Laguerre polynomials, \emph{J. Comput. Appl. Math.} \textbf{313} (2017), 307--317.
\bibitem{AR}{A. Al-Salman, M.B.H. Rhouma, A.A. Al-Jarrah}, On integrals and sums involving special functions, \emph{Missouri J. Math. Sci.} \textbf{23} (2011), 123--141. 
\bibitem{B} P.  Billingsley, Probability and measure., 3rd, Wiley series in Probability and Mathematical Statistics, 1995. 
\bibitem{CR}{J. Choi, A.K. Rathie}, On a hypergeometric summation theorem due to Quereshi et al., \emph{Commun. Korean Math. Soc.} \textbf{28} (2013), 527--534.
\bibitem{BO}{M. Bender, A. Orszag}, \emph{Advanced mathematical methods for scientists and engineers}, McGraw-Hill Inc., New York, 1978. 
\bibitem{KS}A. A. Kilbas, H.M. Srivastava, J.J. Trujillo,  \emph{Theory and Applications of Fractional Differential Equations}. Elsevier, Amsterdam, 2006.
\bibitem{K}{D. Kumar}, Certain integrals of generalized hypergeometric and confluent  functions, \emph{Sigmae, Alfenas},\textbf{5} (2016), 8--18. 
\bibitem{ML} F. Mainardi, Y. Luchko, G. Pagnini, The fundamental solution of the space-time fractional diffusion equation. \emph{Fractional Calculus and Appl. Anal.} \textbf{4} (2001), 153--192.
\bibitem{M}{S. Mishra}, Integrals involving Hermite polynomials, generalized hypergeometric series and Fox's H-function,
and Fourier--Hermite series for products of generalized hypergeometric functions, \emph{Annales
Polonici Mathematici LVI.} \textbf{1} (1991), 19--28. 
\bibitem{N1}{V. Nijimbere}, Evaluation of the non-elementary integral $\int e^{\lambda x^\alpha} dx$, $\alpha\ge2$, and other related integrals, \emph{Ural Math. J.} \textbf{3} (2017), 130--142.
\bibitem{N2}{V. Nijimbere}, Evaluation of some non-elementary integrals involving sine, cosine, exponential and logarithmic integrals: Part I, \emph{Ural Math. J.} \textbf{4} (2018), 24--42. 
\bibitem{N3}{V. Nijimbere}, Evaluation of some non-elementary integrals involving sine, cosine, exponential and logarithmic integrals: Part II,\emph{ Ural Math. J.} \textbf{4} (2018), 43--55. 
\bibitem{N4}{V. Nijimbere}, Analytical and asymptotic evaluations of Dawson's integral and related functions in mathematical physics, \emph{J. Appl. Anal.} \textbf{25} (2018) 43--55. 
\bibitem{N5}{V. Nijimbere}, Asymptotic approximation of the eigenvalues and the eigenfunctions for the Orr-Sommerfeld equation on infinite intervals. \emph{Advances in Pure Mathematics} \textbf{9} (2019), 967 --989. 
\bibitem{NI} NIST Digital Library of Mathematical Functions. \url{http://dlmf.nist.gov/}
\bibitem{QQ}{M.I. Qureshi, K.A. Quraishi, H.M. Srivastava}, Some hypergeometric summation formulas and series identities associated with exponential and trigonometric functions. \emph{Integral Transforms Spec. Funct.} \textbf{19} (2008), 267 --276.
\bibitem{SH} J. Schmid, S. Henningson,  \emph{ Stability and Transitions in Shear Flows}, Springer, New York, 2001.


\end{thebibliography}
\end{document}